\documentclass[12pt]{amsart}
\usepackage{amscd}
\usepackage{amsmath}
\usepackage{verbatim}
\usepackage[normalem]{ulem}
\usepackage{longtable}

\usepackage[all, cmtip,arrow]{xy}

\usepackage{pb-diagram,pb-xy}

\usepackage{amssymb}

\numberwithin{equation}{section}

\newtheorem{thm}[equation]{Theorem}
\newtheorem{prop}[equation]{Proposition}
\newtheorem{lem}[equation]{Lemma}
\newtheorem{cor}[equation]{Corollary}

\theoremstyle{definition}
\newtheorem{rem}[equation]{Remark}
\newtheorem{example}[equation]{Example}
\newtheorem{dfn}[equation]{Definition}
\newtheorem{assumption}[equation]{Assumption}
\newtheorem{ntt}[equation]{}

\newcommand{\codim}{\operatorname{codim}}

\newcommand{\SB}{\mathop{\mathrm{SB}}}
\newcommand{\OGr}{\mathrm{OGr}}

\newcommand{\rk}{\operatorname{rk}}
\newcommand{\CH}{\mathop{\mathrm{CH}}\nolimits}

\newcommand{\SO}{\operatorname{\mathrm{SO}}}

\newcommand{\Spin}{\operatorname{\mathrm{Spin}}}

\newcommand{\PGL}{\operatorname{\mathrm{PGL}}}

\newcommand{\Gm}{\operatorname{\mathbb{G}_m}}

\newcommand{\E}{\mathrm{E}}
\newcommand{\sing}{\mathrm{sing}}

\newcommand{\Ch}{\mathop{\mathrm{Ch}}\nolimits}

\newcommand{\res}{\mathop{\mathrm{res}}\nolimits}

\newcommand{\pr}{\operatorname{\mathrm{pr}}}

\newcommand{\mult}{\operatorname{mult}}

\newcommand{\Char}{\mathop{\mathrm{char}}\nolimits}

\newcommand{\id}{\mathrm{id}}

\newcommand{\imm}{\mathrm{Im}}
\newcommand{\zz}{\mathbb{Z}}
\newcommand{\laz}{\mathbb{L}}
\newcommand{\F}{\mathrm{F}}
\newcommand{\G}{\mathrm{G}}

\newcommand{\B}{\mathrm{B}}

\newcommand{\ff}{\mathbb{F}}

\newcommand{\Spec}{\operatorname{Spec}}
\newcommand{\End}{\operatorname{End}}

\newcommand{\pt}{\mathrm{pt}}

\newcommand{\Ker}{\operatorname{Ker}}

\newcommand{\D}{\mathrm{D}}

\title
[Hopf-theoretic approach to motives]
{Hopf-theoretic approach to motives of twisted flag varieties}

\author
[Victor Petrov]
{Victor Petrov}

\author
[Nikita Semenov]
{Nikita Semenov}

\address{Semenov:
Mathematisches Institut der Universit\"at M\"unchen, Theresienstr. 39, D-80333 M\"unchen, Germany}
\email{semenov@math.lmu.de}

\address{Petrov: St.~Petersburg State University, 29B Line 14th (Vasilyevsky Island), 199178, St.~Petersburg, Russia}
\email{victorapetrov@googlemail.com}


\thanks{The first author was supported by Laboratory of Modern Algebra and Applications, St.~Petersburg State University via grant of the government of the Russian Federation for the state support
of scientific research carried out under the supervision of leading scientists, agreement 14.W03.31.0030 dated 15.02.2018, by Young Russian Mathematics award and by RFBR grant 18-31-20044.
The second author acknowledges the support of the SPP 1786 ``Homotopy theory and algebraic geometry'' (DFG)}

\begin{document}

\begin{abstract}
Let $G$ be a split semisimple algebraic group over a field and let $A^*$ be an oriented cohomology theory in the sense of Levine--Morel.
We provide a uniform approach to the $A^*$-motives of geometrically cellular smooth projective $G$-varieties based on the Hopf algebra structure of $A^*(G)$. Using this approach we provide various applications to the structure of motives of twisted flag varieties.
\end{abstract}

\maketitle

\section{Introduction}

\subsection{Overview of motives}\label{sec-overview}

Chow motives were introduced by Alexander Grothendieck in the 1960s, and they have since
become a fundamental tool for investigating the structure of algebraic varieties.
Computing Chow motives has also proved to be valuable for addressing questions
on other topics: for example, Voevodsky's proof of the Milnor conjecture relies on Rost's computation of
the motive of a Pfister quadric. More generally, the structure of Chow motives of norm varieties plays a crucial role in the proof of the Bloch--Kato conjecture by Rost and Voevodsky.

Applications of Chow motives include, among others, results
on higher Witt indices of quadratic forms \cite{Ka03}, structure of the powers of the fundamental ideal
in the Witt ring \cite{Ka04}, cohomological invariants of algebraic groups \cite{GPS16}, \cite{S16}, Kaplansky's problem on the $u$-invariants of fields \cite{Vi07}, and isotropy of involutions \cite{KaZ13}.

Chernousov, Gille, and Merkurjev established the structure of Chow motives of isotropic twisted flag varieties in \cite{CGM05}. Later Brosnan provided motivic decompositions of twisted flag varieties which are homogeneous under an {\it isotropic} group $G$ (see \cite{Br05}).
Petrov, Semenov, and Zainoulline
established the structure of Chow motives of generically split twisted flag varieties and introduced an invariant of algebraic groups,
called the {\it $J$-invariant} (see \cite{PSZ08}, \cite{PS10}, \cite{PS12}). In the case of quadratic forms an equivalent invariant was introduced previously by Vishik in \cite{Vi05}.
The $J$-invariant allowed, in particular, to construct a new cohomological invariant for groups of type $\mathrm{E}_8$ and
to solve a problem of Serre about groups of type $\mathrm{E}_8$ and
its finite subgroups (see \cite{GS10} and \cite{S16}).

Besides, Garibaldi, Petrov and Semenov used decompositions of Chow motives to relate the rationality of some parabolic subgroups of groups of type $\E_7$ with the Rost invariant, proving a conjecture of Rost and solving a question of Springer in \cite{GPS16}.

There exist also applications of motives with respect to other oriented cohomology theories to algebraic groups.
For example, Panin \cite{Pa94} related the $K^0$-motives of twisted flag varieties and Tits algebras of algebraic groups, generalizing some result of Quillen and Swan.

More recently, Sechin and Semenov used Morava motives to obtain new estimates on torsion in the Chow groups of quadrics (see \cite{SeS20}).

\subsection{}

Let $G$ be a split semisimple algebraic group over a field and let $A^*$ be an oriented cohomology theory in the sense of Levine--Morel. In the present article we provide a uniform approach to the $A^*$-motives of geometrically cellular smooth projective $G$-varieties based on the Hopf algebra structure of $A^*(G)$. Using this approach we provide various applications to the structure of motives of twisted flag varieties.

In particular,

$\bullet$ we show in Theorem~\ref{comod-preserving} that the rationality of cycles is compatible with the comodule structure induced by the Hopf algebra structure of $A^*(G)$;

$\bullet$ we provide in Theorem~\ref{mainthm} a motivic decomposition of the $A^*$-motive of the variety of Borel subgroups for theories $A^*$ satisfying certain axioms;

$\bullet$ we give in Theorem~\ref{RpE} a combinatorial criterion when the Chow motive of a twisted flag variety contains as a direct summand the upper motive of the variety of Borel subgroups;

$\bullet$ we present in Theorem~\ref{conn-quad} new connections in the Chow motives of quadrics which were previously unknown.

We provide now an overview of results of the present article in more details.

\subsection{Unification of motives for different oriented cohomology theories}

Let $A^*$ be an oriented cohomology theory in the sense of Levine--Morel \cite{LM07}, let $G$ be a split semisimple algebraic group over a field $F$, let $B$ be a Borel subgroup of $G$, let $W$ be the Weyl group of $G$, and let $E$ be a $G$-torsor over $F$.

In this article we provide a uniform approach to the $A^*$-motives of twisted flag varieties based on the {\it Hopf algebra} structure of $A^*(G)$.

For example, it is known that the Chow motive of $E/B$ modulo a prime number $p$ is a direct sum of Tate twists of the same indecomposable motive $R_p(E)$, whose structure is described in \cite{PSZ08} in terms of the $J$-invariant of $E$ (see Section~\ref{jinvold}). The rank of $R_p(E)$ can be also expressed in terms of the $J$-invariant and is usually big.

On the other hand, the $K^0$-motive of $E/B$ is a direct sum of $|W|$ indecomposable motives, which in general are not Tate twists of each other. These motives are related to the Tits algebras of $E$ as described in \cite{Pa94} and all of them have rank one.

In the context of the algebraic cobordism theory of Levine--Morel both the Chow theory and the $K^0$-theory are free oriented cohomology theories arising from the same construction with respect to an additive or a multiplicative formal group law. Therefore, at first glance it seems very surprising that the structure of the Chow motives and of $K^0$-motives of $E/B$ are so different.

We provide an explanation of this phenomenon in terms of the coproduct structure of $A^*(G)$. Note that the coproduct structures of $K^0(G)$ and of $\CH^*(G)$ are different even for groups of small rank, like $\PGL_2$.

Moreover, our approach allows to give a definition of the {\it $J$-invariant} for an arbitrary oriented cohomology theory $A^*$ satisfying certain axioms.
We define the $J$-invariant as a quotient of the bialgebra $A^*(G)$ by a certain concrete bi-ideal, which depends on the torsor $E$ (see Definition~\ref{defj}). For example, this bi-ideal is zero if the torsor $E$ is generic. In the case of the Chow motives this definition is equivalent to the old one given in \cite{PSZ08}.

Furthermore, it turns out that the motivic decomposition of $E/B$ with respect to a theory $A^*$ has two layers. The first layer is determined by the $J$-invariant and the second layer is determined by the structure of finitely generated projective modules over the dual algebra (in the sense of Hopf algebras) of the $J$-invariant. This second layer is empty for Chow motives (and therefore remained hidden), but, for example, it is not empty for $K^0$-motives. This provides a conceptual explanation why, opposite to the case of Chow motives, there can be substantially different isomorphism classes of indecomposable direct summands of the $K^0$-motive of $E/B$.

Note that the most of our results can be also applied to arbitrary twisted flag varieties, not necessarily of the form $E/B$ and some our results can be applied more generally to twisted forms of arbitrary cellular varieties equipped with an action of the group $G$ (see e.g. Theorem~\ref{comod-preserving}).

\subsection{Applications to Chow motives}
\subsubsection{Excellent connections of Vishik}
In his celebrated article \cite{Vi11} Vishik shows that the Chow motive of an arbitrary anisotropic quadric $Q$ of dimension $n$ over a field $F$ has
at least the same connections as an anisotropic excellent quadric $P$ of the same dimension as $Q$.
More precisely, Vishik defines an invariant of $Q$ called the motivic decomposition type. Namely, the Chow motive of $Q$ splits over $\overline F$ as a direct sum of Tate motives $M(\overline Q)\simeq\oplus_{\lambda\in\Lambda(Q)}\zz\{\lambda\}$, where $$\Lambda(Q)=\{0,1,\ldots,[n/2]\}\sqcup\{n-[n/2],\ldots,n-1,n\}.$$ If $N$ is a direct summand of the Chow motive $M(Q)$ over $F$, then the motive $N$ splits over $\overline F$ as $\oplus_{i\in\Lambda(N)}\zz\{i\}$
for some $\Lambda(N)\subset\Lambda(Q)$, and one says that $\lambda,\mu\in\Lambda(Q)$ are connected, if for every direct summand $N$ of the Chow motive $M(Q)$ over $F$ one has that either both $\lambda,\mu\in\Lambda(N)$ or both $\lambda,\mu\not\in\Lambda(N)$ (see Definition~\ref{dfnvishik}).

Vishik shows in \cite{Vi11} that if $\lambda,\mu$ are connected in the Chow motive of an anisotropic {\it excellent} quadric $P$ of dimension $n$, then they are connected in the Chow motive of every anisotropic $n$-dimensional quadric over $F$. Since the motivic decompositions of excellent quadrics are known, this provides explicit restrictions on the motivic decomposition type of quadrics. This result has further applications discussed in \cite{Vi11}.

In the present article using our approach we provide new connections in the Chow motives of quadrics which were previously unknown (see Theorem~\ref{conn-quad}). Our connections are usually complementary to Vishik's excellent connections and thus, combining both of them one gets stronger restrictions on the motivic decomposition type of quadrics. Note that Vishik's approach to excellent connections relies on the Steenrod operations. Thus, one can view the coproduct structure as a complementary tool to Steenrod operations.

\subsubsection{$J$-invariant, motivic decompositions and rational cycles for the Chow theory}
In the case of the Chow motives
our approach to the $J$-invariant is more conceptual than in \cite{PSZ08} and, in our opinion, the proof of the motivic decomposition of $E/B$ (Corollary~\ref{maincor}) is more simple than the original proof given in \cite{PSZ08}.

We show that the realizations of rational cycles respect the coproduct structure of the $J$-invariant, and
using this we provide new motivic decompositions of Chow motives of (not necessarily generically split) twisted flag varieties including all varieties of type $\E_8$ at the prime $3$, and we obtain new restrictions on rational cycles and on the $J$-invariant (see Section~\ref{sec-except}).

In fact, our method leads to a simplified proof of the Rost conjecture for groups of type $\E_7$ (see Section~\ref{sec-overview} above). The crucial point in the proof of this conjecture was Lemma~10.8 of \cite{GPS16}, where we did extensive computations using Steenrod operations. This lemma was used once in \cite{GPS16} to compute the Chow motives of some $\E_7$-varieties. We do not reprove Rost's conjecture in this article (since it is already proved in \cite{GPS16}), but as an illustration we prove Proposition~\ref{dec-e7} which contains one of the motivic decompositions needed for the proof of Rost's conjecture. Other necessary motivic decompositions can be obtained in a similar manner avoiding the use of Lemma~10.8 of \cite{GPS16}.

\subsubsection{Applications to upper motives}
In \cite{Ka13} Karpenko introduced the notion of upper motives and proved that any indecomposable direct summand of the Chow
motive of a twisted flag variety of inner type is isomorphic to a Tate twist of the
upper motive of another twisted 
flag variety. Thus, the study of motivic decompositions
of twisted 
flag varieties is reduced to the study of the upper motives.

In Theorem~\ref{RpE} we provide a necessary and sufficient criterion when the Chow motive of a twisted flag variety contains as a direct summand the upper motive of the variety of Borel subgroups. We also use this criterion in Section~\ref{sec-except} to compute the Chow motives of some exceptional varieties.

\subsection{Coproduct structure of the Chow theory of algebraic groups}
There exists an extensive literature mostly of a Japanese mathematical school devoted to computations of the coproduct structure on $H^*_{\sing}(G)$ for a split semisimple complex group $G$ (see e.g. \cite{IKT76}, \cite{KM77}, \cite{MT78}, \cite{MZ77}).

In this article we provide in Section~\ref{sec-coac} a new method to compute the coproduct structure for $\CH^*(G)$, where $G$ is a split semisimple group over an arbitrary field (of an arbitrary characteristic).

Namely, in Sections~\ref{sec-gen}, \ref{sec-quad} and \ref{sec-except} we get formulae for the coproduct using motivic decompositions of twisted flag varieties which are homogeneous under an {\it isotropic} group. Motivic decompositions in this situation are  given in \cite{CGM05} and \cite{Br05}. For computations we use graphical interpretation of these decompositions based on cutting the Hasse diagrams along edges described in \cite{Se07}. Various Hasse diagrams are provided in the Appendix of \cite{PSV98}.

\subsection{Applications to Morava motives}

The $K^0$-motives of twisted flag varieties were computed by Panin in  \cite{Pa94} generalizing previous results of Quillen \cite[Section~8]{Qui73} and Swan \cite[Theorem~1]{Sw85}. 
In Section~\ref{secother}
we recover decompositions of the
$K^0$-motives of some twisted flag varieties using the method of the present article.

Finally, we illustrate in Section~\ref{secother} the methods developed in this article by calculating
the Morava motives of some twisted 
flag varieties for which this computation
was previously not possible.

{\bf Acknowledgements.} We would like to thank Stefan Gille, Alexander Henke, and Alexander Vishik for discussions and e-mail conversations on the subject of the article.

\section{Background on oriented cohomology theories and motives}

Consider an oriented cohomology theory $A^*$ in the sense of Levine--Morel over a field $F$ (see \cite{LM07}). Throughout the article we assume that the theory $A^*$ is generically constant (see \cite[Definition~4.4.1]{LM07}) and satisfies the localization property \cite[Definition~4.4.6]{LM07}. Moreover,  we assume in the article that $\mathrm{char}\,F=0$ except when $A^*$ is the Chow theory in which case the characteristic of $F$ can be arbitrary.

If $\Char F=0$, we consider the algebraic cobordism $\Omega^*$ of Levine--Morel. By \cite[Theorem~1.2.6]{LM07} the algebraic cobordism is a universal oriented cohomology theory, i.e. for every oriented cohomology theory $A^*$ over $F$ there
is a (unique) morphism of theories 
$\Omega^* \to A^*$.

Each oriented cohomology theory $A^*$ is equipped with a $1$-dimensional commutative
formal group law. For example, for the Chow theory $\CH^*$ this is the additive formal
group law, for Grothendieck's $K^0[\beta,\beta^{-1}]$ this is the multiplicative formal group law and for $\Omega^*$ the universal formal
group law.

The canonical morphism $\mathbb{L}\to\Omega^*(\Spec F)$ from the Lazard ring is an isomorphism (see \cite[Theorem~1.2.7]{LM07}).
It is well known that $\laz\simeq\zz[t_1,t_2,\ldots]$ with $\deg t_i=-i$.

\begin{dfn}[Levine--Morel, {\cite[Remark~2.4.14(2)]{LM07}}]
Let $S$ be a commutative
ring, let $\mathcal{F}_S$ be a formal group law over $S$ and let $\laz \to S$ be the respective ring morphism.
Then 
$\Omega^* \otimes_\laz S$ is an oriented cohomology theory which is called a {\it free
theory}. Its ring of coefficients is $S$, and its associated formal group law is $\mathcal{F}_S$.
\end{dfn}

For example, the Chow theory and $K^0[\beta,\beta^{-1}]$ are free theories (see \cite[Theorem~1.2.18 and 1.2.19]{LM07}).

\begin{ntt}[Morava $K$-theory]
If $\Char F=0$, we consider for a prime number $p$ and a natural number $n$ the $n$-th Morava $K$-theory $K(n)^*$ with respect to $p$.
Notice that we do not include $p$ in the notation.
We define this theory as a free theory
with the coefficient ring $\ff_{p}[v_n,v_n^{-1}]$ where $\deg v_n=-(p^n-1)$
and with a formal group law of height $n$.

If $n=1$, there exists a functorial (with respect to pullbacks) isomorphism of algebras
$K(1)^*(-)/(v_1-1)\simeq K^0(-)\otimes\ff_{p}$, which can be 
obtained with the help of the Artin--Hasse exponent.

By \cite[Appendix 2]{Ra} there is a split surjective graded ring homomorphism $$\varphi\colon\laz_{(p)}\to\zz_{(p)}[v_1,v_2,\ldots]$$ with $\deg v_i=-(p^i-1)$ which classifies the formal group laws which are $p$-typical. In particular, one can consider $v_i$ as an element in $\laz_{(p)}$. Moreover, the composition
$$\laz_{(p)}\xrightarrow{\varphi}\zz_{(p)}[v_1,v_2,\ldots]\to\ff_p[v_n,v_n^{-1}],$$
where the second map is the canonical projection followed by a localization,
defines the formal group law for the Morava $K$-theory (sometimes also called the Honda formal group law).
\end{ntt}

\begin{ntt}[Motives]
For a theory $A^*$ we consider the category of {\it $A^*$-motives}
which is defined exactly in the same way as the category of Grothendieck's Chow motives with $\CH^*$
replaced by $A^*$
(see \cite{Ma68}, \cite{EKM}). In particular, the morphisms between two smooth
projective irreducible varieties $X$ and $Y$ over $F$ are given by
$A^{\dim Y}(X\times Y)$.

We denote the motive of a smooth projective variety $X$ over a field $F$ by $M(X)$, and we write $A^*(\pt)$ for the motive of $\pt=\Spec F$. For a motive $M$ we denote its Tate twists by $M\{m\}$.
\end{ntt}

\begin{ntt}[Rost Nilpotence]\label{rostnilp}
Let $A^*$ be an oriented cohomology theory and consider the category of $A^*$-motives over $F$.
Let $M$ be an $A^*$-motive over $F$. We say that the Rost nilpotence principle holds for $M$,
if the kernel of the restriction homomorphism $$\End(M)\to\End(M_L)$$ consists of nilpotent correspondences for all field extensions $L/F$.

By \cite[Section~8]{CGM05} Rost nilpotence holds for the Chow motives of all twisted flag varieties and by \cite[Corollary~3.5]{GiV18}
Rost nilpotence holds for the $A^*$-motives of all twisted flag varieties for every free oriented cohomology theory $A^*$.

\end{ntt}

\begin{ntt}[Cellular varieties] 
In this article we consider smooth projective cellular varieties over a field $F$. We say that a smooth projective variety $X$ is cellular, if it possesses a stratification by closed (not necessarily smooth) subvarieties $$\emptyset=X_{-1}\subset X_0\subset\ldots\subset X_n=X$$ such that for all $i$ there is an isomorphism $X_i\setminus X_{i-1}\simeq\mathbb{A}^{n_i}$ for some $n_i\ge 0$.

Let $A^*$ be an oriented cohomology theory satisfying the localization axiom. Then the $A^*$-motive of $X$ is a direct sum of Tate motives (see \cite{NZ06}). Moreover, the K\"unneth formula holds for $X$. Namely, if $Y$ is an arbitrary smooth variety, then $$A^*(X\times Y)\simeq A^*(X)\otimes_{A^*(\pt)} A^*(Y).$$
\end{ntt}

\section{Hopf-theoretic background}

Let $G$ be a split semisimple algebraic group over a field $F$, let $T$ be a split maximal torus of $G$, let $B$ be a Borel subgroup of $G$ containing $T$ and let $A^*$ be an oriented cohomology theory. It is well-known that the multiplication in $G$ induces the structure of a commutative, graded Hopf algebra on $A^*(G)$ over $A^*(\pt)$.

\begin{ntt}[Ring $A^*(G)$]\label{sec31}
If $A^*$ is an oriented cohomology theory, there exists an algorithm to compute the ring structures of $A^*(G)$ and of $A^*(G/B)$. Indeed, it suffices to determine the ring structure for $A^*=\Omega^*$.

First of all, since the variety $G/B$ is cellular, $\Omega^*(G/B)$ is a free $\laz$-module. Its free generators can be parametrized by the elements of the Weyl group $W$ of $G$. More precisely, for each $w\in W$ one fixes its reduced word decomposition and associates with it a certain class $Z_w\in\Omega^{l(w)}(G/B)$, where $l(w)$ denotes the length of $w$, which is a Bott--Samelson resolution of singularities of a Schubert subvariety of $G/B$ corresponding to $w$ (see \cite{CPZ13}). The class $Z_w$ depends on a particular choice of a reduced decomposition of $w$, but abusing notation we omit it in our notation. Then $Z_w$, $w\in W$, form a free basis of $\Omega^*(G/B)$.

Let $\B T$ denote the classifying space of $T$. There is a characteristic map
\begin{align}
c\colon\Omega^*(\B T)\to\Omega^*(G/B)
\end{align}
which is a ring homomorphism. Besides, the pullback of the canonical projection ${G\to G/B}$ induces a ring homomorphism $\pi\colon \Omega^*(G/B)\to \Omega^*(G)$.

It follows from \cite[Proposition~5.1]{GiZ12}
that the sequence
$$\Omega^*(\B T)\xrightarrow{c}\Omega^*(G/B)\xrightarrow{\pi} \Omega^*(G)$$ of graded rings
is right exact (i.e. $\pi$ is surjective and its kernel is the ideal of $\Omega^*(G/B)$ generated by the elements of positive degrees in the image of $c$). Then the explicit combinatorial description of the map $c$ given in \cite{CPZ13} allows to compute explicitly the ring structure of $\Omega^*(G)$.

In particular, since $\Omega^*(G/B)$ has finite rank over $\laz$ (namely, the rank equals $|W|$), the module $\Omega^*(G)$ is finitely generated over $\laz$ (and hence $A^*(G)$ is a finitely generated $A^*(\pt)$-module for every oriented cohomology theory $A^*$).

We remark, however, that the existing algorithms are not feasible for explicit computations for groups of a big rank.

Nevertheless, one can find in the literature an explicit description of the ring structure of $A^*(G)$ for some oriented cohomology theories $A^*$ and some groups $G$. For example, Merkurjev computes in \cite{Me97}
$K^0(G)$ for all split semisimple groups $G$, and Yagita provides some computations of algebraic cobordism in \cite{Ya05} (see Section~\ref{secother} below for some concrete examples).
\end{ntt}

\begin{ntt}[Structure of Hopf algebras]\label{strHopf}
A celebrated theorem of Borel asserts that every commutative, graded (by the non-negative integers) connected finite dimensional bialgebra over a finite field $\ff_p$ is isomorphic as an algebra to $\ff_p[e_1,\ldots,e_r]/(e_1^{p^{k_1}},\ldots,e_r^{p^{k_r}})$ for some integers $r$, $k_i$ and some homogeneous generators $e_i$ (see \cite[Theorem~7.11 and Proposition~7.8]{MM65}).

For example, if $A^*=\Ch^*:=\CH^*\otimes\ff_p$
is the Chow ring modulo $p$, then this agrees with formulae for $\Ch^*(G)$ from \cite{Kac85}.
\end{ntt}

\begin{ntt}[$J$-invariant for Chow motives]\label{jinvold}
For a fixed prime $p$ we denote by $\Ch^*:=\CH^*\otimes\ff_p$ the Chow ring modulo $p$.
Let $G$ be a split semisimple algebraic group over a field $F$, $B$ a Borel subgroup of $G$ and $E$ a $G$-torsor over $F$. Then $$\Ch^*(G)\simeq \ff_p[e_1,\ldots,e_r]/(e_1^{p^{k_1}},\ldots,e_r^{p^{k_r}})$$
for some integers $r$, $k_i$ and with $\deg e_i=:d_i$. We assume that the sequence of $d_i$ is non-decreasing.

We introduce an order on the set of additive generators 
of $\Ch^*(G)$, i.e., on the monomials $e_1^{m_1}\ldots e_r^{m_r}$. 
To simplify the notation, we denote the monomial
$e_1^{m_1}\ldots e_r^{m_r}$ by $e^M$, where $M$ is an $r$-tuple
of integers $(m_1,\ldots,m_r)$. The codimension (in the Chow ring) of
$e^M$ is denoted by $|M|$. Observe that $|M|=\sum_{i=1}^rd_im_i$.

Given two $r$-tuples $M=(m_1,\ldots,m_r)$ and $N=(n_1,\ldots,n_r)$ we say
$e^M\le e^N$ (or equivalently $M\le N$) if either $|M|<|N|$, or $|M|=|N|$ and 
$m_i\le n_i$ for the greatest $i$ such that $m_i\ne n_i$.
This gives a well-ordering on the set of all monomials ($r$-tuples).

\begin{dfn}[{\cite[Definition~4.6]{PSZ08}}]\label{def71}
Denote as $\overline{\Ch}^*(G)$ the image of the composite map
$$
\Ch^*(E/B)\xrightarrow{\res} \Ch^*(G/B) \xrightarrow{\pi} \Ch^*(G),
$$
where $\pi$ is the pullback of the canonical projection $G/B\to G$ and $\res$ is the scalar extension to a splitting field of the torsor $E$.

For each $1\le i\le r$ set $j_i$ to be the smallest non-negative
integer such that the subring $\overline{\Ch}^*(G)$ contains an element $a$ 
with the greatest monomial $e_i^{p^{j_i}}$ 
with respect to the order on $\Ch^*(G)$ as above, i.e.,
of the form 
$$
a=e_i^{p^{j_i}}+\sum_{e^M\lneq e_i^{p^{j_i}}} c_M e^M, \quad c_M\in\ff_p.
$$
The $r$-tuple of integers $(j_1,\ldots,j_r)$ is called
the {\it $J$-invariant} of $E$ modulo $p$ and is denoted by $J(E)$ or $J_p(E)$.
Note that $j_i\le k_i$ for all $i$.
\end{dfn}

By \cite{PSZ08} the Chow motive of $E/B$ with coefficients in $\ff_p$ decomposes in a direct sum of Tate twists of an indecomposable motive $R_p(E)$, and the Poincar\'e polynomial of $R_p(E)$ over a splitting field of $E$ equals
\begin{equation}\label{fpoin}
\prod_{i=1}^r\frac{t^{d_ip^{j_i}}-1}{t^{d_i}-1},
\end{equation}
where $(j_1,\ldots,j_r)$ is the $J$-invariant of $E$.

In the case of quadratic forms, i.e. when $G$ is a special orthogonal group and $E$ is a $G$-torsor, there is an equivalent notion of the $J$-invariant introduced by Vishik in \cite{Vi05}. Namely, the torsor $E$ corresponds to a non-degenerate quadratic form $q$ with trivial discriminant of dimension $2m+1$ or $2m+2$. The maximal orthogonal Grassmannian $\OGr(\max,\bar q)$ of the quadratic form $\bar q=q\times_F \overline F$ has for $p=2$ certain concrete generators ${z_k\in\Ch^k(\OGr(\max,\bar q))}$, $k=0,\ldots,m$, and the $J$-invariant $J(q)$ of $q$ is defined as the set of those $k\in\{0,\ldots,m\}$ such that $z_k$ are rational, and $0\not\in J(q)$, if $\dim q$ is odd (see \cite[Definition~5.11]{Vi05}, \cite[\S88]{EKM}).
Besides, one has $J(q)=\begin{cases}
\{1,\ldots,m\}\setminus J', & \text{if }\dim q=2m+1,\\
\{0,\ldots,m\}\setminus J', & \text{if }\dim q=2m+2,
\end{cases}$

\noindent
where $J'=\{2^ld_i\mid i=1,\ldots,r;\, 0\le l\le j_i-1\}$ with $d_i=2i-1$, $r=[\frac{m+1}{2}]$ and $J_2(E)=(j_1,\ldots,j_r)$. These formulae allow to switch between different definitions of the $J$-invariant in the case of quadratic forms.
\end{ntt}

\begin{ntt}[Demazure operators]\label{secdem}
In this section we follow \cite{CPZ13} (cf. \cite{De73}, \cite{De74}).
Let $E$ be a $G$-torsor over $\Spec F$. For each simple root $\alpha_i$ consider the natural projection $\pi_i\colon E/B\to E/P_{\{i\}}$, where $P_{\{i\}}$ is the parabolic subgroup corresponding to the root $\alpha_i$. Set $\kappa_i=G(x_{\alpha_i},x_{-\alpha_i})\in A^*(\B T)$,
where
$$
G(x,y)=\frac{x+y-\mathcal{F}_A(x,y)}{xy},
$$
$\mathcal{F}_A$ is the formal group law of the theory $A^*$, and $x_\alpha$ is the image of the generator of $A^1(\B\Gm)$ under the map
$$
A^*(\B\Gm)\to A^*(\B T)\to A^*_T(E)\simeq A^*(E/B),
$$
where the first map is induced by $\alpha$.

Define the operator $C_i$ on $A^*(E/B)$ to be the composition $\pi_i^*(\pi_i)_*$ and $\Delta_i$ by the formula $\Delta_i(x)=\kappa_i x-C_i(x)$. Note that $C_i$ and $\Delta_i$ commute with morphisms of oriented theories. Denote by $\tilde\varepsilon$ the pullback map to the generic point $A^*(E/B)\to A^*(\pt)$.

In the particular case of a trivial torsor we have the following results.

\begin{lem}\label{lemDiff}
The operators $s_i=\id-x_{\alpha_i}\Delta_i$ are ring homomorphisms defining a representation of the Weyl group on $A^*(G/B)$ over $A^*(\pt)$, and the following Leibniz rule holds:
$$
\Delta_i(uv)=\Delta_i(u)v+s_i(u)\Delta_i(v).
$$
\end{lem}
\begin{proof}
It suffices to show the formula for the algebraic cobordism, and moreover, since $\Omega^*(G/B)$ is torsion free, for $A^*=\Omega^*\otimes_\zz \zz[t^{-1}]$, where $t$ is the torsion index of $G$. In this situation the characteristic map
\begin{equation}\label{charmap}
c\colon A^*(\B T)\simeq A^*_T(\pt)\to A^*_T(G)\simeq A^*(G/B)
\end{equation}
is surjective (see \cite[Corollary~13.10]{CPZ13}), and it is enough to verify the formula at the level of $A^*(\B T)$. But the action of $s_i$'s on $A^*(\B T)$ coincides with the usual Weyl group action (cf. \cite[Definition~3.5]{CPZ13}), and the formula is verified in \cite[Proposition~3.8]{CPZ13}.
\end{proof}

The operators $s_i$ from the above lemma are called simple reflections.

\begin{lem}\label{lemDem}
Let $x$ be an element from $A^*(G/B)$. If for every sequence $i_1,\ldots,i_n$ with $n\le\dim G/B$ we have
$$
\tilde\varepsilon\circ\Delta_{i_1}\circ\ldots\circ\Delta_{i_n}(x)=0,
$$
then $x=0$.
\end{lem}
\begin{proof}
For every $w\in W$ fix its reduced decomposition $I_w=(i_1,\ldots,i_{l(w)})$ such that ${w=s_{i_1}\ldots s_{i_{l(w)}}}$. Note that $l(w)\le l(w_0)=\dim G/B$, where $w_0$ is the longest element in the Weyl group. Set $\Delta_{I_w}=\Delta_{i_1}\circ\ldots\circ\Delta_{i_{l(w)}}$. By \cite[Proposition~5.4 and Theorem~13.13]{CPZ13} there is a basis $\zeta_w$ in $\Omega^*(G/B)$ over $\Omega^*(\pt)$ such that $\tilde\varepsilon\circ\Delta_{I_w}(\zeta_v)=\delta_{vw}$ for all $v,w\in W$. Since $\Omega^*$ is universal, a basis with the same properties exists for every $A^*$, and the claim follows.
\end{proof}
\end{ntt}

\section{General Hopf-theoretic statements}

For an arbitrary smooth algebraic group $G$ over a field $F$ we say that $E$ is a right $G$-torsor over a smooth $F$-variety $X$, if $E$ is a scheme over $X$ equipped with a right action of $G_X=G\times_F X$ which is locally trivial in \'etale topology.

Let $G$ be a split semisimple algebraic group over a field $F$, $T$ a split maximal torus of $G$ over $F$ and $B$ a Borel subgroup of $G$ over $F$ containing $T$.

\begin{lem}\label{T-tors}
Let $E$ be a right $T$-torsor over a smooth variety $X$. Then
$$A^*(E)\simeq A^*(X)\otimes_{A^*(\B T)}A^*(\pt).$$
\end{lem}
\begin{proof}
Similar to the proof of \cite[Proposition~5.1]{GiZ12}. Namely, let $\chi_1,\ldots,\chi_n$ stand for a basis of the character group of $T$, and let $L(\chi_i)$ denote ${\mathbb A}^1$ with the action of $T$ via $\chi_i$. Then $T$ embeds $T$-equivariantly  into $\oplus_i L(\chi_i)\simeq {\mathbb A}^n$ with the complement equal to the union of the coordinate hyperplanes. Twisting by $E$ produces an open embedding of $E$ into a vector bundle $V={}_E{\mathbb A}^n_X$ over $X$ with the complement equal to the union of zero loci of $\chi_i$'s considered as regular maps on $V$. Applying homotopy invariance and the localization axiom we get an isomorphism
$$
A^*(E)\simeq A^*(X)/(c^A_1({}_EL(\chi_1)),\ldots,c^A_1({}_EL(\chi_n))),
$$
as claimed.
\end{proof}

\begin{cor}\label{G-tors}
Let $E$ be a right $G$-torsor over a smooth variety $X$. Then
$$A^*(E)\simeq A^*(E/B)\otimes_{A^*(\B T)}A^*(\pt).$$
\end{cor}
\begin{proof}
Indeed, $E$ is a $T$-torsor over $E/T$, and the natural
map $E/T\to E/B$ is an affine bundle, hence gives an
isomorphism $A^*(E/B)\simeq A^*(E/T)$, and it remains to
apply Lemma~\ref{T-tors}.
\end{proof}

\begin{lem}\label{Kunn}
Let $X$ be a smooth variety. Then
$$A^*(X\times G)\simeq A^*(X)\otimes_{A^*(\pt)}A^*(G).$$
\end{lem}
\begin{proof}
Since $X\times G$ is a trivial $G$-torsor over $X$, we have
\begin{align*}
A^*(X\times G)&\simeq A^*(X\times G/B)\otimes_{A^*(\B T)}A^*(\pt)\text{ by Corollary~\ref{G-tors}}\\
&\simeq A^*(X)\otimes_{A^*(\pt)}A^*(G/B)\otimes_{A^*(\B T)}A^*(\pt)\text{ since $G/B$ is cellular}\\
&\simeq A^*(X)\otimes_{A^*(\pt)}A^*(G)\text{ by Corollary~\ref{G-tors}.}
\end{align*}
\end{proof}

In general, for a variety $X$ (in this article we consider either geometrically cellular varieties $X$ or $X=E$) we denote by $\overline X$ the extension of scalars to a splitting field of $X$ and by
$\res\colon A^*(X)\to A^*(\overline X)$ the restriction homomorphism. For $\alpha\in A^*(X)$ we denote by $\bar\alpha$ the image of $\alpha$ under $\res$.

The $A^*(\pt)$-algebra $A^*(G)$ is a Hopf algebra, where the coproduct homomorphism $\Delta$ is the pullback of the multiplication in the group $G$ followed by the isomorphism of Lemma~\ref{Kunn}, i.e. $\Delta\colon A^*(G)\xrightarrow{\text{mult}} A^*(G\times G)\simeq A^*(G)\otimes_{A^*(\pt)}A^*(G)$.

The counit of $A^*(G)$ is the pullback of the embedding of the identity element of $G$ into $G$ and the antipode is the pullback of the inverse map for the group $G$.

Let $E$ be a right $G$-torsor over $F$.
Over a splitting field of $E$ the torsor $\overline E$ becomes isomorphic to $\overline G$. Since $G$ is a split group and the theory $A^*$ is generically constant, the Hopf algebra $A^*(G)$ does not depend
on the base field and the restriction homomorphism identifies $A^*(G)$ and $A^*(\overline G)$.
In particular, we have an isomorphism ${A^*(\overline E)\simeq A^*(G)}$.

Moreover, we claim that this isomorphism does not depend on the choice of an identification between $E$ and $G$ over a splitting field. Indeed, we can assume that the splitting field, say $K$,
is separably closed and consider the simply connected cover of $G$ over $K$. An isomorphism between $G$ and $E$ over $K$ corresponds to the choice of a $K$-rational point of $G$.
Since the simply connected cover of $G$ over $K$ is generated by its root elements $x_\alpha(\xi)$, we can consider the family $x_{\alpha}(\xi t)$ which gives a homotopy between the identity element (for $t=0$)
and $x_\alpha(\xi)$ (for $t=1$). Therefore, the homotopy invariance of the theory $A^*$ implies the claim.

Thus, we can consider the restriction homomorphism as ${\res\colon A^*(E)\to A^*(G)}$.

Then the diagram
\begin{equation}\label{tors-res}
\xymatrix{
A^*(E)\ar[r]\ar[d]_{\res} & A^*(E)\otimes_{A^*(\pt)}A^*(G)\ar^{\res\otimes\id}[d]\\
A^*(G)\ar^-{\Delta}[r] & A^*(G)\otimes_{A^*(\pt)}A^*(G)
}
\end{equation}
commutes.

\begin{lem}\label{bi-ideal}
The ideal $J$ in the algebra $A^*(G)$ generated by $$\imm(A^*(E)\xrightarrow{\res} A^*(G))\cap\Ker(A^*(G)\xrightarrow{\varepsilon} A^*(\pt)),$$
where $\varepsilon$ is the counit, is a two-sided bi-ideal in the bialgebra $A^*(G)$.
\end{lem}
\begin{proof}
By definition $J$ is contained in $\Ker\varepsilon$. Take any element $\bar e$ from $$\imm(A^*(E)\xrightarrow{\res} A^*(G))\cap\Ker(A^*(G)\xrightarrow{\varepsilon} A^*(\pt)).$$ By diagram~\eqref{tors-res} we can write $\Delta(\bar e)=1\otimes{\bar e}+\sum\bar e_i\otimes a_i$ with
$\bar e_i$ from $$\imm(A^*(E)\xrightarrow{\res} A^*(G))\cap\Ker(A^*(G)\xrightarrow{\varepsilon} A^*(\pt))$$ and $a_i$ from $A^*(G)$. But this sum belongs to $A^*(G)\otimes_{A^*(\pt)}J+J\otimes_{A^*(\pt)}A^*(G)$. Since $\Delta$ is a ring homomorphism, the lemma follows.
\end{proof}

\begin{dfn}\label{defj}
Define the bialgebra $H^*:=A^*(G)/J$. We call $H^*$ the {\it $J$-invariant} of $E$ with respect to the theory $A^*$.
\end{dfn}

\begin{example}
If $E$ is a standard generic torsor (see \cite[Section~3]{PS17}), then ${A^*(E)=A^*(\pt)}$ for every free theory $A^*$ (cf. \cite[Lemma~3.1]{PS17}). Therefore, in this case $H^*=A^*(G)$.
\end{example}

\begin{rem}\label{jinv-same}
Let us show that in the case of the Chow theory modulo a prime $p$ the bialgebra $H^*$ contains essentially the same information as the tuple of integers $(j_1,\ldots,j_r)$ from Definition~\ref{def71}.

Note that by Corollary~\ref{G-tors} we have a commutative diagram with surjective vertical arrows
$$
\xymatrix{
\Ch^*(E/B)\ar[r]^-\res\ar@{->>}[d]&\Ch^*(G/B)\ar@{->>}[d]\\
\Ch^*(E)\ar[r]^-\res&\Ch^*(G).
}
$$
This means that the image of the bottom map that appears in Definition~\ref{defj} coincides with the image of the map from the top left corner to the bottom right corner that appears in Definition~\ref{def71}.

Factorization modulo $J$ is a surjective algebra homomorphism
\begin{equation}\label{fff1}
\varphi\colon\Ch^*(G)\simeq\ff_p[e_1,\ldots,e_r]/(e_1^{p^{k_1}},\ldots,e_r^{p^{k_r}})\to H^*\simeq\ff_p[f_1,\ldots,f_s]/(f_1^{p^{l_1}},\ldots,f_s^{p^{l_s}})
\end{equation}
(see Subsection~\ref{strHopf}). We may assume $l_1\ge l_2\ge\ldots\ge l_s>0$. The map $\varphi$ induces a surjective linear map of vector spaces
$$
\varphi^+\colon\Ch^{>0}(G)/(\Ch^{>0}(G))^2\simeq\langle e_1,\ldots,e_r\rangle\to H^{>0}/(H^{>0})^2\simeq\langle f_1,\ldots f_s\rangle.
$$
Note that a homogeneous linear {\it upper triangular} substitutions of generators of $H^*$ does not change the relations: for example, if $\codim f_1=\codim f_2$, then we still have ${(f_1+f_2)^{p^{l_1}}=0}$ and $f_2^{p^{l_2}}=0$.
On the other hand, a homogeneous linear {\it lower triangular} substitution of generators of $\Ch^*(G)$ does not change the value of the $J$-invariant from Definition~\ref{def71}.

We claim that one can recover the presentation of $H^*$ as in formula~\eqref{fff1} from the $J$-invariant of $E$ (in the sense of Definition~\ref{def71}) and, conversely, one can recover the $J$-invariant of $E$ of Definition~\ref{def71} from the homomorphism~$\varphi$. Note that this is immediate, if the codimensions of the generators $e_i$ are pairwise distinct.

By \cite[Section~3]{Mal10} every matrix over a field has a LEU-decomposition, which can be seen as a generalized version of the Bruhat decomposition. Namely, every matrix can be written in the form $LEU$, where $L$ is a lower triangular matrix, $U$ is an upper triangular matrix and $E$ is a truncated permutation matrix.

Applying this decomposition to the matrix of $\varphi^+$
we can thus adjust the generators $e_i$, $f_j$ in such a way that $\varphi^+$ sends each $e_i$ to either $0$ or $f_{m_i}$ for some $m_i$. Since $\varphi^+$ is surjective, $l_{m_i}$'s are determined by $j_i$'s . Conversely, we can restore the value of $j_i$ as $0$ or $l_{m_i}$ respectively.
\end{rem}

\begin{dfn}
Let $X$ be a smooth projective geometrically cellular variety over $F$ equipped with a left action of $G$.
Define the structure of a left $H^*$-comodule on $A^*(\overline X)$ as the composition
$$
\rho\colon A^*(\overline X)\to A^*(\overline G\times\overline X)\to A^*(G)\otimes_{A^*(\pt)}A^*(\overline X)\to H^*\otimes_{A^*(\pt)}A^*(\overline X),
$$
where the first map is the pullback of the action of $\overline G$ on $\overline X$. We call $\rho$ the coaction map.

Note that $\rho$ preserves multiplication, as all intermediate maps do.
\end{dfn}

\begin{lem}\label{pushf}
The $H^*$-comodule structure is compatible with pullbacks and pushforwards along equivariant projective morphisms.
\end{lem}
\begin{proof}
For pullbacks the claim is obvious, for $\rho$ is defined in terms of pullbacks. For a projective morphism $f\colon X\to Y$ we have a Cartesian square
$$
\xymatrix{
{\overline G}\times{\overline X}\ar[r]\ar_-{\id\times{\bar f}}[d] & {\overline X}\ar^-{\bar f}[d]\\
{\overline G}\times{\overline Y}\ar[r]\ar[r] &{\overline Y},
}
$$
whose horizontal maps are flat. Therefore, this square is transversal (see \cite[Definition~1.1.1]{LM07}) and, thus, induces the following commutative diagram in cohomology:
$$
\xymatrix{
A^*({\overline G}\times{\overline X})\ar_-{(\id\times{\bar f})_*}[d] & A^*({\overline X})\ar^-\rho[l]\ar^-{\bar f_*}[d]\\
A^*({\overline G}\times{\overline Y})&A^*({\overline Y})\ar^-\rho[l],
}
$$
and the claim follows.
\end{proof}

Let $X$ be a smooth projective $G$-variety over $F$. We denote by ${}_E X$ the twisted form of $X$ by means of the torsor $E$. Note that ${}_E X$ and $X$ are isomorphic over every splitting field of $E$.

\begin{lem}\label{comod-triv}
Let $X$ be smooth projective cellular $G$-variety over $F$ and 
let $\alpha$ be an element from $A^*({}_E X)$.
Then $\bar\alpha$ is a coinvariant element, i.e. $\rho(\bar\alpha)=1\otimes\bar\alpha$ in $H^*\otimes_{A^*(\pt)}A^*(\overline X)$.
\end{lem}
\begin{proof}
Consider the natural projection map $E\times X\to{}_E X$. Then the following diagram commutes:
$$
\xymatrix{
A^*({}_E X)\ar[r]\ar_{\res}[d] & A^*(E)\otimes_{A^*(\pt)}A^*(\overline X)
\ar^{\res\otimes\id}[d]\\
A^*(\overline X)\ar[r] & A^*(G)\otimes_{A^*(\pt)}A^*(\overline X).
}
$$
Now we can write $\rho(\bar\alpha)=1\otimes{\bar\alpha}+\sum{\bar e_i\otimes a_i}$ with $\bar e_i$ from $$\imm(A^*(E)\xrightarrow{\res} A^*(G))\cap\Ker(A^*(G)\xrightarrow{\varepsilon} A^*(\pt))$$ and $a_i$ from $A^*(\overline X)$, and the claim follows.
\end{proof}

\begin{dfn}\label{realiz}
Let $X$ and $Y$ be smooth projective varieties over $F$.
Any element $\alpha$ from $A^*(X\times Y)$ defines the \emph{realization map}
$$
\alpha_\star=(\pr_{Y})_*\circ\mu_\alpha\circ\pr_X^*\colon A^*(X)\to A^*(Y),
$$
where $\mu_\alpha$ stands for the multiplication by $\alpha$.
\end{dfn}

\begin{thm}\label{comod-preserving}
Let $X$, $Y$ be smooth projective cellular $G$-varieties over $F$. Let ${}_E X$, ${}_E Y$ be the respective twisted forms of $X$ and $Y$. Let $\alpha$ be a correspondence from $A^*({}_E X\times {}_E Y)$. Then $\bar\alpha_\star\colon A^*(\overline X)\to A^*(\overline Y)$ is a homomorphism of $H^*$-comodules.
\end{thm}
\begin{proof}
We have to show that the diagram
$$
\xymatrix{
A^*(\overline X)\ar^-{\rho}[r]\ar_{\bar\alpha_\star}[d] & H^*\otimes_{A^*(\pt)}A^*(\overline X)
\ar^{\id\otimes\bar\alpha_\star}[d]\\
A^*(\overline Y)\ar^-{\rho}[r] & H^*\otimes_{A^*(\pt)}A^*(\overline Y)
}
$$
commutes. We have
\begin{align*}
\rho\circ\bar\alpha_\star&=\rho\circ(\pr_{\bar Y})_*\circ\mu_{\bar\alpha}\circ\pr_{\bar X}^*\text{ by Definition~\ref{realiz}}\\
&=(\id\otimes(\pr_{\bar Y})_*)\circ\mu_{\rho(\bar\alpha)}\circ(\id\otimes\pr_{\bar X}^*)\circ\rho\text{ by Lemma~\ref{pushf}}\\
&=(\id\otimes(\pr_{\bar Y})_*)\circ(\id\otimes\mu_{\bar\alpha})\circ(\id\otimes\pr_{\bar X}^*)\circ\rho\text{ by Lemma~\ref{comod-triv}}\\
&=(\id\otimes\bar\alpha_\star)\circ\rho.
\end{align*}
\end{proof}

\begin{rem}
Theorem~\ref{comod-preserving} means that there is a ``realization'' functor from the category of motives of $E$-twisted forms of cellular $G$-varieties to the category of graded $H^*$-comodules. It is easy to see that the functor preserves the tensor product structure, where as usual the tensor product of comodules is considered as a comodule via the multiplication map:
$$
\rho_{M\otimes N}=\mult\circ\big(\rho_M\otimes\rho_N\big).
$$
\end{rem}

\section{Generically split twisted flag varieties}\label{secgensplit}

\begin{assumption}\label{assump}
Starting from this section we assume that every finitely generated graded $A^*(\pt)$-module is projective.
\end{assumption}

For example, this assumption holds for Chow groups modulo $p$ and for the Morava $K$-theory modulo $p$.

Let $E$ be a right $G$-torsor over $F$.
Define the subring of ``rational cycles''
$$
R^*=\imm\big(A^*(E/B)\xrightarrow{\res} A^*(G/B)\big).
$$

Furthermore, define the ideal $$I=R^*\cap\Ker(A^*(G/B)\xrightarrow{\widetilde\varepsilon} A^*(\pt))$$ in $R^*$, where $\widetilde\varepsilon$ is the pullback map to a rational point. Note that the diagram
$$
\xymatrix{
A^*(G/B)\ar[rd]^-{\widetilde\varepsilon}\ar[r]^-\pi&A^*(G)\ar[d]^-\varepsilon\\
&A^*(\pt),
}
$$
where $\varepsilon$ is the counit map and $\pi$ is the pullback of the canonical projection $G\to G/B$, is commutative.

\begin{lem}\label{nilpideal}
The ideal $I$ is nilpotent.
\end{lem}
\begin{proof}
Consider the expansion of an element $a$ from $\Ker(A^*(G/B)\xrightarrow{\widetilde\varepsilon} A^*(\pt))$ in the standard basis $Z_w$ (see Section~\ref{sec31}). Note that $$\widetilde\varepsilon(Z_w)=\begin{cases}
1, & w=1;\\
0, & w\ne 1.
\end{cases}$$
Therefore, $\tilde\varepsilon(a)=0$
is the coefficient of the element $Z_1\in A^*(G/B)$.
All other elements $Z_w$, $w\ne 1$, from the
standard basis are of positive codimension and come from the algebraic cobordism $\Omega^*(G/B)$ and, hence, are nilpotent. Therefore, the ideal $I$ is nilpotent.
\end{proof}

\begin{lem}\label{modI}
$A^*(G/B)/IA^*(G/B)\simeq H^*$ as $A^*(\pt)$-algebras.
\end{lem}
\begin{proof}
We have
\begin{align*}
A^*(G)&\simeq A^*(G/B)\otimes_{A^*(\B T)}A^*(\pt)\text{ by Corollary~\ref{G-tors}}\\
&\simeq A^*(G/B)\otimes_{A^*(E/B)}\big(A^*(E/B)\otimes_{A^*(\B T)}A^*(\pt)\big)\\
&\simeq A^*(G/B)\otimes_{A^*(E/B)}A^*(E)\text{ by Corollary~\ref{G-tors}}.
\end{align*}
It follows that
$$
A^*(G)\otimes_{A^*(E)}A^*(\pt)\simeq A^*(G/B)\otimes_{A^*(E/B)}A^*(\pt),
$$
where the right-hand side is $A^*(G/B)/IA^*(G/B)$, and the left-hand side is $H^*$.
\end{proof}

By Assumption~\ref{assump} there exists a section
$$\sigma\colon H^*\to A^*(G/B)$$ of the canonical homomorphism $A^*(G/B)\to H^*$ of $A^*(\pt)$-modules.

\begin{lem}\label{mult-iso}
The map
\begin{align*}
\theta\colon H^*\otimes_{A^*(\pt)}R^*&\to A^*(G/B)\\ x\otimes y&\mapsto \sigma(x)y
\end{align*}
is an isomorphism of $R^*$-modules.
\end{lem}
\begin{proof}
Indeed, by Lemma~\ref{modI} the map $\theta$ becomes an isomorphism
after tensoring with $A^*(\pt)$ over $R^*$. But this
tensoring is the same as taking the quotient modulo the ideal $I$. So, the cokernel $M$ of $\theta$ satisfies $MI=M$, and $I$ is nilpotent by Lemma~\ref{nilpideal}, so $M=0$.

To prove the injectivity of $\theta$ we localize the coefficient ring $A^*(\pt)$ at a prime ideal and so assume that $H^*$ is free. Choose a basis $e_1,\ldots e_n$ of $H^*$.

Note that $R^*$ is stable under the Demazure operators $\Delta_i$ (see Section~\ref{secdem}) because they are defined on $A^*(E/B)$. Consider the filtration on $R^*$ whose $k$-th term $R^{(k)}$ is the intersection of the kernels of all linear functions of the form $\tilde\varepsilon\circ A_1\circ\ldots\circ A_l$ such that each $A_j$ is either $\Delta_{i_j}$ or $s_{i_j}$, and $\Delta$'s appear at most $k$ times.

By Lemma~\ref{lemDem} it suffices to show that if $a\in H^*\otimes_{A^*(\pt)} R^{(k)}$ lies in the kernel of $\theta$, then $a\in H^*\otimes_{A^*(\pt)} R^{(k+1)}$. Write $a$ as $\sum e_i\otimes y_i$ with $y_i\in R^{(k)}$. Now applying $D=A_1\circ\ldots\circ A_l$ containing $k+1$ $\Delta$'s to the image of $a$, using Lemma~\ref{lemDiff} $l$ times and collecting the leftmost summands of each sum separately we get an expression of the form
$$
\sum_i w(\sigma(e_i))D(y_i)+\sum_{i,j} D'_j(\sigma(e_i))D_j(y_i),
$$
where each $D_j$ contains at most $k$ $\Delta$'s, $D'_j$ are of the same form as $D$, i.e., $D'_j$ is a composition of Demazure operators and simple reflections, and $w$ is a product of some simple reflections. On the other hand, this expression must be zero. Taking it modulo $I$ and applying $w^{-1}$ we see that each $D(y_i)$ belongs to $I$ and so $y_i$ belongs to $R^{(k+1)}$.
\end{proof}

\begin{lem}\label{cofree}
There is a (non-canonical and perhaps distinct from the isomorphism $\theta$ of Lemma~\ref{mult-iso}) isomorphism of $H^*$-comodules between $A^*(G/B)$ and the cofree comodule $H^*\otimes_{A^*(\pt)}R^*$.
\end{lem}
\begin{proof}
Recall that a cofree comodule over an $A^*(\pt)$-module $R^*$ is the module $H^*\otimes_{A^*(\pt)}R^*$ with the coaction $H^*\otimes_{A^*(\pt)}R^*\xrightarrow{\Delta\otimes\id} H^*\otimes_{A^*(\pt)}H^*\otimes_{A^*(\pt)}R^*$.

Lemma~\ref{comod-triv} implies that the $I$-adic filtration on $A^*(G/B)$ is compatible with the coaction by $H^*$. There is an isomorphism
$$
H^*\otimes_{A^*(\pt)}\big(I^kR^*/I^{k+1}R^*\big)\to I^kA^*(G/B)/I^{k+1}A^*(G/B)
$$
induced by the map $\theta$ of Lemma~\ref{mult-iso}. It is compatible with the coaction (where the left-hand side is considered as a cofree $H^*$-comodule), since the diagram
$$
\xymatrix{
A^*(G/B)\ar[r]\ar[d] & A^*(G)\otimes_{A^*(\pt)}A^*(G/B)\ar[d]\\
A^*(G)\ar[r]\ar[d] & A^*(G)\otimes_{A^*(\pt)}A^*(G)\ar[d]\\
H^*\ar[r] & H^*\otimes_{A^*(\pt)}H^*
}
$$
commutes, and the kernel of the composite vertical map on the left is generated by $I$ by Lemma~\ref{modI}. Since the quotients of the filtration are cofree, it (non-canonically) splits, so we have an isomorphism of comodules
$$A^*(G/B)\simeq H^*\otimes_{A^*(\pt)}\Big(\bigoplus_{k\ge 0} I^kR^*/I^{k+1}R^*\Big).$$
But by Assumption~\ref{assump} $R^*\simeq \bigoplus_{k\ge 0} I^kR^*/I^{k+1}R^*$ as $A^*(\pt)$-modules, and the claim follows.
\end{proof}

\begin{lem}\label{mainlem} We have the following ring isomorphism:
$$\imm\big(A^*(E/B\times E/B)\xrightarrow{\res} A^*(G/B\times G/B)\big)\simeq H^\vee\otimes_{A^*(\pt)}\End_{A^*(\pt)}R^*,$$
where the multiplication in the ring on the left-hand side is given by the composition product and $H^\vee$ is the $A^*(\pt)$-algebra dual to $H^*$.
\end{lem}
\begin{proof}
Theorem~\ref{comod-preserving} implies that the image of $A^*(E/B\times E/B)$ in
$$
A^*(G/B\times G/B)\simeq\End_{A^*(\pt)}(G/B)
$$
actually lies in $\End_{H^*}A^*(G/B)$. The latter is isomorphic to $H^\vee\otimes_{A^*(\pt)}\End_{A^*(\pt)}R^*$ by Lemma~\ref{cofree}, and since by Assumption~\ref{assump} all $A^*(\pt)$-modules are projective, it remains to compare the ranks over $A^*(\pt)$. But it follows from \cite{CM06} and \cite{NZ06} that
$$
A^*(E/B\times E/B)\simeq A^*(E/B)\otimes_{A^*(\pt)}A^*(G/B)
$$
as $A^*(E/B)$-modules, and the isomorphism is compatible with the restriction map. So the rank of the image is equal to
$$
\rk_{A^*(\pt)}R^*\cdot\rk_{A^*(\pt)}A^*(G/B)=(\rk_{A^*(\pt)}R^*)^2\cdot\rk_{A^*(\pt)}H^*\text{ by Lemma~\ref{mult-iso}},
$$
which is the same as the rank of $H^\vee\otimes_{A^*(\pt)}\End_{A^*(\pt)}R^*$.
\end{proof}

\begin{thm}[Two layers of motivic decompositions]\label{mainthm}
In the above notation assume additionally that the theory $A^*$ is free. Then there is a one-to-one correspondence between motivic decompositions of the $A^*$-motive of $E/B$ and direct sum decompositions of $A^*(G/B)$ as an $H^\vee$-module.

Assume further that $R^*$ is graded free as an $A^*(\pt)$-module:
\begin{equation}\label{decompR}
R^*\simeq\bigoplus_{i\in\mathcal{I}} A^*(\pt)(i)
\end{equation}
for some multiset of non-negative integers $\mathcal{I}$. Then there exists an $A^*$-motive $\mathcal R$ such that the $A^*$-motive of $E/B$ decomposes as follows:
\begin{equation}\label{decompGB}
M(E/B)\simeq\bigoplus_{i\in\mathcal{I}}\mathcal{R}\{i\}.
\end{equation}

There is a one-to-one correspondence between motivic decompositions of $\mathcal R$ in the category of $A^*$-motives and direct sum decompositions of $H^\vee$ as a module over itself.

Moreover, $A^*(\overline{\mathcal R})\simeq H^*$ as $H^*$-comodules.
\end{thm}
\begin{proof}
By Rost Nilpotence (see Subsection~\ref{rostnilp}) motivic decompositions of $M(E/B)$ correspond to full systems of mutually orthogonal idempotents in
$$
\imm\big(A^*(E/B\times E/B)\xrightarrow{\res} A^*(G/B\times G/B)\big).
$$
Applying Lemma~\ref{mainlem} and the Morita equivalence we see that these decompositions correspond to direct sum decompositions of $H^*\otimes_{A^*(\pt)}R^*\simeq A^*(G/B)$ as an $H^\vee$-module.

The second claim is clear now: decomposition~\eqref{decompR} of $R^*$ as an $A^*(\pt)$-module implies the respective decomposition of $A^*(G/B)$ as an $H^\vee$-module and hence decomposition~\eqref{decompGB}. Since $H^*$ is finitely generated projective over $A^*(\pt)$, the structure of an $H^\vee$-module determines the structure of an $H^*$-comodule and vice versa. So, decompositions of $H^*$ as an $H^\vee$-module correspond to decompositions of $H^\vee$ as a module over itself. Finally, $A^*(\overline{\mathcal R})$ is isomorphic to $H^*$ as an $H^\vee$-module, so it is isomorphic to $H^*$ as an $H^*$-comodule.
\end{proof}

\begin{rem}
Condition~\eqref{decompR} seems difficult to check. However, assume additionally that all graded modules of constant rank over $A^*(\pt)$ are free (and not only projective) and assume that $H^*$ is also graded free over $A^*(\pt)$. Then Lemma~\ref{mult-iso} and a rank counting imply that the module $R^*$ is graded free as well.
\end{rem}

Finally, we give a Hopf-theoretic proof of \cite[Theorem~5.13]{PSZ08}.

\begin{cor}\label{maincor}
Let $p$ be a prime number and $A^*=\CH^*\otimes\ff_p$.

Then the respective motive $\mathcal R$  from Theorem~\ref{mainthm} is indecomposable and its Poincar\'e polynomial over a splitting field of $E$ is given by formula~\eqref{fpoin}.
\end{cor}
\begin{proof}
We have $H^0=\ff_p$, so there are no non-trivial decompositions of $H^\vee$ as an $H^\vee$-module. By the last assertion of Theorem~\ref{mainthm} we see that the Poincar\'e polynomial of $\mathcal R$ is the same as the Poincar\'e polynomial of $H^*$, which in view of Remark~\ref{jinv-same} is given by formula~\eqref{fpoin}.
\end{proof}

\section{Applications to Chow motives: generalities}\label{sec-gen}

\begin{ntt}[Computing coaction]\label{sec-coac}
We use the following method for computing the coaction of $\CH^*(G)$ on $\CH^*(G/P)$ for a parabolic subgroup $P$ of $G$. Choose a parabolic subgroup $Q$ in $G$ and denote the commutator subgroup of its Levi subgroup by $C$. By \cite[Lemma~2.3 and Lemma~2.4]{PS12} $\CH^*(C)$ is a quotient of $\CH^*(G)$. More precisely, we have the following result.

\begin{lem}\label{chow-par}
$\CH^*(C)\simeq\CH^*(G)\otimes_{\CH^*(G/Q)}\zz$.
\end{lem}
\begin{proof}
By \cite[Proposition~1]{EG97} applied to the map $G/B\to G/Q$ we have
$$
\CH^*(G/B)\simeq\CH^*(G/Q)\otimes_{\zz}\CH^*(Q/B)
$$
as $\CH^*(G/Q)$-modules, so
$$
\CH^*(Q/B)\simeq\zz\otimes_{\CH^*(G/Q)}\CH^*(G/B).
$$
Moreover, the respective map from $\CH^*(G/B)$ to $\CH^*(Q/B)$ coincides with the restriction to the fiber over the generic point of the projection map from $G/B$ to $G/Q$ and so preserves multiplication. Tensoring with $\zz$ over $\CH^*(\B T)$ and using Corollary~\ref{G-tors} and \cite[Lemma~2.3]{PS12} we obtain the result.
\end{proof}

Consider a generic $C$-torsor $E$ (in the sense that $\CH^*(E)=\zz$; cf. \cite[Lemma~3.1]{PS17}) and the respective variety ${}_E(G/P)$, where $G/P$ is considered as a $C$-variety. Note that the group $_EG$ is isotropic. Then by \cite{CGM05} and \cite{Br05} the Chow motive of ${}_E(G/P)$ decomposes as a direct sum of Tate twists of motives of projective homogeneous {$C$-va\-rie\-ties}, and by Theorem~\ref{comod-preserving} this decomposition is compatible with the coaction
$$
\CH^*(G/P)\to\CH^*(C)\otimes\CH^*(G/P),
$$
so we can compute the coaction modulo the kernel of the natural map $\CH^*(G)\to\CH^*(C)$ once we know the coaction for projective $C$-homogeneous varieties. Since the rank of $C$ is strictly smaller than the rank of $G$, this gives an inductive procedure to compute the coaction. We will illustrate this method in the proofs of Lemmas~\ref{quad-eq}, \ref{E7P7mod2} and \ref{E8mod3} below.
\end{ntt}

\begin{ntt}
Recall that for a fixed prime $p$ we write $\Ch^*$ instead of $\CH^*\otimes\ff_p$.
By Corollary~\ref{maincor} the Chow motive of $E/B$ modulo $p$ decomposes into a direct sum of indecomposable motives which are isomorphic to non-negative Tate twists of an indecomposable motive which we denote by $R_p(E)$ (see also Section~\ref{jinvold}). We also have
$$
H^*\simeq\ff_p[e_1,\ldots,e_r]/(e_1^{p^{j_1}},\ldots,e_r^{p^{j_r}}),
$$
where $(j_1,\ldots,j_r)$ is the $J$-invariant of $E$.

If $P$ is a parabolic subgroup, the Chow motive of $E/P$ modulo $p$ can contain a summand isomorphic to a Tate twist of $R_p(E)$ as well. Now we give a combinatorial criterion when this happens.

\begin{thm}\label{RpE}
Let $P$ be a parabolic subgroup of a split semisimple algebraic group $G$ over a field $F$ and let $E$ be a $G$-torsor over $F$. Denote by $\rho$ the coaction of $H^*$ on $\Ch^*(G/P)$.

Every summand of the Chow motive $M(E/P)$ with coefficients $\ff_p$ which is isomorphic to a Tate twist of $R_p(E)$ has a generic point $\alpha\in\Ch^*(G/P)$ such that for some $\beta\in\Ch^*(G/P)$ we have
$$
\rho(\beta)=e_J\otimes\alpha+\sum a_i\otimes b_i
$$
for some $a_i$, $b_i$
with $\codim a_i<\codim e_J$, where $e_J=e_1^{p^{j_1}-1}\cdots e_r^{p^{j_r}-1}$.

Conversely, for every $\beta$ of this form there is a summand of the Chow motive $M(E/P)$ with coefficients $\ff_p$ which is isomorphic to a Tate twist of $R_p(E)$ and whose generic point is $\alpha$.
\end{thm}
\begin{proof}
Assume that the motive $R_p(E)\{m\}$ is a direct summand of the motive $M(E/P)$ for some integer $m$.
Consider the following diagram which is commutative by Theorem~\ref{comod-preserving}
$$
\xymatrix{
\Ch^*(\overline{R_p(E)})\ar[d]^-{\iota}\ar[r]&H^*\otimes\Ch^*(\overline{R_p(E)})\ar[d]^-{\id\otimes\iota}\\
\Ch^*(M(G/P)\{-m\})\ar^-{\rho}[r]&H^*\otimes\Ch^*(M(G/P)\{-m\})
}
$$
where $\iota$ is induced by 
the embedding of the motivic summand $R_p(E)\to M(E/P)\{-m\}$.

By the last assertion of Theorem~\ref{mainthm} we can identify $H^*\simeq \Ch^*(\overline{R_p(E)})$.

We can take $\alpha$ to be the image of $1$ and $\beta$ to be the image of $e_J$ under the comodule map $H^*\xrightarrow{\iota}\Ch^*(M(G/P)\{-m\})$. This implies the first claim.

To prove the converse statement consider the following commutative diagram:
$$
\xymatrix{
\Ch^*_T(G/P)\ar^-{\pr_{G/P}^*}[r]\ar@{->>}[d]&\Ch^*_T(E\times G/P)\ar[r]\ar[d]&\Ch^*_T(E\times E/P)\simeq\Ch^*(E/B\times E/P)\ar[d]\\
\Ch^*(G/P)\ar@{=}[d]\ar^-{\pr_{G/P}^*}[r]&\Ch^*(E\times G/P)\ar[r]\ar^{\res}[d]&\Ch^*(E\times E/P)\ar^{\res}[d]\\
\Ch^*(G/P)\ar[r]^-{\pr_{G/P}^*}&\Ch^*(G\times G/P)\ar[r]&\Ch^*(G\times G/P)\simeq \Ch^*(G)\otimes_{\ff_p}\Ch^*(G/P).
}
$$

The torus $T$ acts on $G/P$ naturally on the left, on $E\times G/P$ by the rule ${t(e,gP)=(et^{-1},tgP)}$ and on $E\times E/P$ on the first factor only. The last action is obviously free, and so the $T$-equivariant Chow ring can be identified with the Chow ring of the quotient (see \cite[Proposition~8(a)]{EG98}). The vertical maps from the top to the middle row are the forgetful maps, and the vertical maps from the middle row to the bottom row are the restriction maps.

The second map in each row is induced by the isomorphism $E\times G/P\to E\times E/P$ (resp. $G\times G/P\to G\times G/P$ for the bottom row) sending $(e,gP)$ to $(e,egP)$ (resp. $(h,gP)$ to $(h,hgP)$). An immediate check shows that these maps are $T$-equivariant (note that we let $T$ act on $G\times G/P$ in two different ways). The composite bottom map is induced by the map $G\times G/P\to G/P$ sending $(h,gP)$ to $h^{-1}gP$ and so coincides with $(S\otimes\id)\circ\tilde\rho$, where $S$ denotes the antipode in the Hopf algebra $\Ch^*(G)$ and $\tilde\rho$ is the coaction map.

For convenience of the reader let us simplify the diagram:
\begin{equation}\label{diagCh}
\xymatrix{
\Ch^*_T(G/P)\ar@{->>}[d]\ar[r]&\Ch^*(E/B\times E/P)\ar[d]\\
\Ch^*(G/P)\ar[r]^-{(S\otimes\id)\circ\tilde\rho}\ar[rd]&\Ch^*(G)\otimes_{\ff_p}\Ch^*(G/P)\ar[d]\\
&H^*\otimes_{\ff_p}\Ch^*(G/P).
}
\end{equation}

We know that in $\Ch^*(G)$
$$
\Delta(e_i)=e_i\otimes 1+1\otimes e_i+\sum u_j\otimes v_j
$$
with $\codim u_j>1$ and $\codim v_j>1$, so
$$
S(e_i)=-e_i+\text{ lower degree terms,}
$$
where the order of terms is as in Definition~\ref{def71}.
By dimensional reasons this implies that the image of $S(e_J)$ in $H^*$ is $\pm e_J$.

Recall that in Section~\ref{secgensplit} we denoted by $\sigma\colon H^*\to \Ch^*(G/B)$ a section of the canonical homomorphism $\Ch^*(G/B)\to H^*$.
Consider now any preimage of $\pm\beta$ in the top left corner. It goes to some cycle $x$ in the top right corner, that is in $\Ch^*(E/B\times E/P)$. Using the commutativity of the diagram and Lemma~\ref{modI} we have
$$
\bar x=\sigma(e_J)\times\alpha+\sum a'_i\times b_i+\sum \delta_i c_i\times d_i
$$
with $\delta_i\in R^{>0}$, some cycles $a'_i$, $c_i$, $d_i$, and the same $b_i$'s as in the statement of the theorem.

By Lemma~\ref{mult-iso} we can write the class of a rational point $[\pt]$ in $\Ch^*(G/B)$ as $\gamma\sigma(e_J)$ for some rational $\gamma\in R^*$ of the maximal possible degree (in particular, $\gamma\delta=0$ for every $\delta\in R^{>0}$). Then
$$
f=(\gamma\times 1)\bar x=[\pt]\times\alpha+\sum a''_i\times b_i
$$
with $\codim a''_i<\dim G/B$ is a rational cycle in $\Ch^*(G/B\times G/P)$.

Denote by $\alpha^\vee$ a Poincar\'e dual to $\alpha$, meaning that $\deg(\alpha\alpha^\vee)=1$. Using \cite[Proposition~1]{EG97} applied to the projection map $E/B\times E/P\to E/B$ (or diagram~\eqref{diagCh} again) we see that there is a rational cycle in $\Ch^*(G/P\times G/B)$ of the form
$$
g=\alpha^\vee\times 1+\sum l_j\times m_j
$$
with $\codim m_j>0$.

Now we have
\begin{align*}
&f\circ g=\alpha^\vee\times\alpha+\sum_{i,j}\deg(a_i''m_j)l_j\times b_i;\\
&g\circ f=[\pt]\times 1+\sum_{i,j}\deg(l_jb_i)a_i''\times m_j.
\end{align*}
Since $\Ch^*(G/B\times G/B)$ and $\Ch^*(G/P\times G/P)$ are finite, there is a positive integer $M$ such that $(f\circ g)^{\circ M}$ and $(g\circ f)^{\circ M}$ are idempotents, and $g\circ (f\circ g)^{\circ(M-1)}$ and $f$ are rational isomorphisms between the corresponding summands. Applying the Rost Nilpotence principle (see Section~\ref{rostnilp}) we get the result.
\end{proof}
\end{ntt}

\section{Applications to Chow motives: quadrics}\label{sec-quad}

Consider the projective quadric $Q$ corresponding to a non-degenerate quadratic form $q$ of rank $n=2m+2$ or $2m+1$ with trivial discriminant. Then ${\overline Q}\simeq\SO_n/P_1$, where $P_1$ is the maximal parabolic subgroup corresponding to the subset $\{2,\ldots,m\}$ of the Dynkin diagram of the respective group (in fact, the results of this section (Lemma~\ref{l1}, Lemma~\ref{quad-eq} and Theorem~\ref{conn-quad}) hold for fields of an arbitrary characteristic, i.e., one can write $\mathrm{O}^+_n$ instead of $\SO_n$). The enumeration of simple roots follows Bourbaki.

\begin{lem}\label{l1}
$\CH^*(\SO_n)\otimes\ff_2\simeq\ff_2[e_1,\ldots,e_m]/(e_i^2=e_{2i})$ with $\codim e_i=i$ if $i\le m$ and $e_i=0$ if $i>m$.
\end{lem}
\begin{proof}
Follows from \cite[Table~II]{Kac85}, cf. \cite[Proposition~3.1]{Vi05}.
\end{proof}

We denote by $h$ the generator in $\Ch^1(\overline Q)$ and by $l$ the generator in $\Ch^m(\overline Q)$ (or one of two generators distinct from $h^m$ in the even-dimensional case).

We denote by $\rho$ the map defining the comodule structure on $\Ch^*(\overline Q)$.

\begin{lem}\label{quad-eq}
$\rho(l)=\sum_{i=1}^m e_i\otimes h^{m-i}+1\otimes l$.
\end{lem}
\begin{proof}
We use the method described in Section~\ref{sec-coac}.
We proceed by induction on $m$. The base $m=1$ is clear: $\SO_n/P_1$ is either the projective line or the product of two projective lines. If $m>1$ consider the parabolic subgroup $Q=P_1$. Let $E$ be a generic $C$-torsor, where $C$ is the commutator subgroup of the Levi subgroup of $P_1$. Then the Chow motive of ${}_E(\SO_n/P_1)$ decomposes as follows (see \cite[Proposition~2]{Ro98}):
$$
M({}_E(\SO_n/P_1))=\ff_2\oplus M({}_E(\SO_{n-2}/P'_1))\{1\}\oplus \ff_2\{n-2\},
$$
where the middle summand is a quadric of smaller dimension (in this summand $P'_1$ stands for the maximal parabolic subgroup of $\SO_{n-2}$ of type $1$). By the induction hypothesis we have
\begin{equation}\label{q1}
\rho(l)=\sum_{i=1}^{m-1} e_i\otimes h^{m-i}+1\otimes l\mod ( e_m)\otimes\Ch^*(\SO_n/P_1),
\end{equation}
where $(e_m)$ is the ideal of $\Ch^*(\SO_n)$ generated by $e_m$.

We have a commutative diagram
\begin{equation}\label{diagso}
\xymatrix{
\Ch^*(\SO_n/P_1)\ar[d]^-{\pi}\ar[r]^-\rho&\Ch^*(\SO_n)\otimes\Ch^*(\SO_n/P_1)\ar[d]^-{\id\otimes\pi}\\
\Ch^*(\SO_n)\ar^-{\Delta}[r]&\Ch^*(\SO_n)\otimes\Ch^*(\SO_n)
}
\end{equation}

The image of $l$ in $\Ch^*(\SO_n)$ is $e_m$, so
\begin{equation}\label{q2}
\rho(l)=e_m\otimes 1+1\otimes l\mod \Ch^*(\SO_n)\otimes \Ker\pi.
\end{equation}

Moreover, $\Ker\pi=(h)$. It remains to combine formulae~\eqref{q1} and \eqref{q2} together.
\end{proof}

\begin{dfn}\label{dfnvishik}
Following \cite[\S4]{Vi04} and \cite{Vi11} we set $\Lambda(Q)=\{0,\ldots,n-2\}$, if $n$ is odd and $\Lambda(Q)=\{0,\ldots,n-2\}\coprod\{m'\}$, if $n$ is even.

For a direct summand $N$ of the Chow motive $M(Q)$ we define $\Lambda(N)$ to be the subset of $\Lambda(Q)$ consisting of all $i$ such that $N$ over a splitting field of $Q$ contains as a direct summand the Tate motive $\ff_2\{i\}$. In the case $i=m$ and $n$ is even we say that $\Lambda(N)$ contains $m'$, if the realization of ${N}$ over a splitting field of $X$ contains $h^m$ and that $\Lambda(N)$ contains $m$, if the realization of ${N}$ over a splitting field contains $l$ or $l+h^m$.

Finally, we say that $M(Q)$ has a \emph{connection} $i$ and $j$, if for every indecomposable direct summand $N$ of $M(Q)$ such that $\Lambda(N)$ contains $i$, it also contains $j$. This is an equivalence relation that defines a partition of $\Lambda(Q)$.
\end{dfn}

The main result of \cite{Vi11} says that every anisotropic quadric has at least the same connections as an anisotropic excellent quadric of the same dimension. Now we state some new restrictions depending on the $J$-invariant $J(q)$, where $J(q)$ stands for Vishik's $J$-invariant. Note that if $G$ is a split special orthogonal group, then $\CH^*(G)\otimes\ff_2$ is isomorphic to the modulo two Chow ring of the respective maximal orthogonal Grassmannian. This identifies the generators $e_i$ with the generators $z_i$  (see Section~\ref{jinvold}).

\begin{thm}\label{conn-quad}
If $j\not\in J(q)$,
then the Chow motive of $Q$ has connections $m-j$ and $m$, $m-j+1$ and $m+1,\ldots, m-1$ and $m+j-1$, and, if $n$ is even, $m'$ and $m+j$.
\end{thm}
\begin{proof}
The condition $j\not\in J(q)$ means that the generator $e_j$ of the orthogonal group is not zero in the bialgebra $H^*$ corresponding to $q$.

Consider an indecomposable summand $N$ of the motive $M(Q)$ whose realization
contains $l$ or $l+h^m$.
By Theorem~\ref{comod-preserving} the realization of $N$ is a subcomodule of $\CH^*(\overline Q)$ under the coaction of $H^*$, so Lemma~\ref{quad-eq} implies that the realization contains $h^{m-j}$. This means that there is a connection $m-j$ and $m$. The other connections can be obtained similarly considering the cycles $h^k l$, $k=1,\ldots, m$, and using the identity $\rho(h^k)=1\otimes h^k$ (see Lemma~\ref{comod-triv}).
\end{proof}

We illustrate the above theorem by several small-dimensional examples (cf. \cite[Section~7]{Vi04}). As before we denote by $q$ a quadratic form and by $Q$ the respective projective quadric. For simplicity we assume that $\Char F\ne 2$.

\begin{example}
Consider an anisotropic quadric $Q$ of dimension $6$ with trivial discriminant. If $1\in J(q)$, then by \cite[Proposition~4.2]{PS10} $Q$ has trivial Clifford invariant, hence is a Pfister quadric. Otherwise by Theorem~\ref{conn-quad} there are connections $2$ and $3$, $3'$ and $4$, and by \cite{Vi11} there are connections $0$ and $3$, $1$ and $4$, $2$ and $5$ and $3'$ and $6$. So there are at most two indecomposable summands shown in the following diagram:
$$
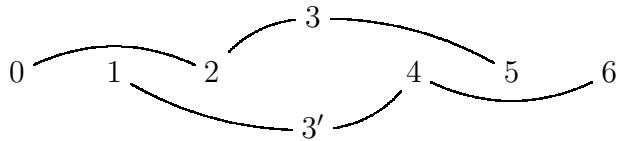
\xymatrix@R-1.5pc{
&&&{3}\ar@/_/@{-}[ld]\ar@/^/@{-}[rrd]\\
{0}\ar@/^2ex/@{-}[rr]&{1}\ar@/_/@{-}[rrd]&{2}&{}&{4}\ar@/^/@{-}[ld]\ar@/_2ex/@{-}[rr]&{5}&{6}\\
&&&{3'}
}
$$
If, moreover, $2\not\in J(q)$, the motive is indecomposable.
\end{example}

\begin{example}
Consider an anisotropic quadric $Q$ of dimension $8$ with trivial discriminant. It is known that $Q$ has non-trivial Clifford invariant, so $1\not\in J(q)$, and by Theorem~\ref{conn-quad} there are connections $3$ and $4$, $4'$ and $5$. On the other hand, by \cite{Vi11} there are connections $0$ and $7$, $1$ and $8$, $2$ and $5$, $3$ and $6$, $4$ and $4'$. So there are at most three indecomposable summands as follows:
$$
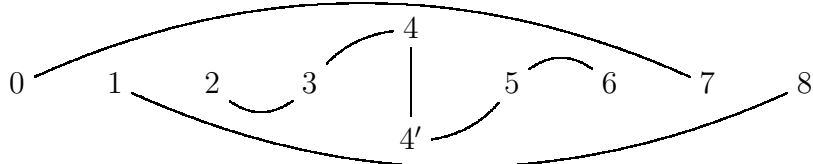
\xymatrix@R-1.5pc{
&&&&{4}\ar@{-}[dd]\ar@/_/@{-}[ld]\\
{0}\ar@/^6ex/@{-}[rrrrrrr]&{1}\ar@/_6ex/@{-}[rrrrrrr]&{2}\ar@/_2ex/@{-}[r]&{3}&{}&{5}\ar@/^2ex/@{-}[r]\ar@/^/@{-}[ld]&{6}&{7}&{8}\\
&&&&{4'}
}
$$
If, moreover, $3\not\in J(q)$, the motive is indecomposable.
\end{example}

\begin{example}
Consider an anisotropic quadric $Q$ of dimension $10$ with trivial discriminant and trivial Clifford invariant.

By \cite[Exercise~88.14]{EKM} or \cite[Corollary~4.2]{QSZ12} $1\in J(q)$, since the Clifford invariant of $q$ is trivial. Therefore, $2\in J(q)$ and $4\in J(q)$, since the respective generators $e_2$ and $e_4$ of $\CH^*(\SO_{12})\otimes\ff_2$ are powers of $e_1$.

If $3\in J(q)$, then by the action of the Steenrod algebra (\cite[Proposition~5.12]{Vi05}, \cite[Last column of Table~4.13]{PSZ08}) $5\in J(q)$ and hence $J(q)$ is trivial, a contradiction. 
So $3\not\in J(q)$, and by Theorem~\ref{conn-quad} there are connections $2$ and $5$, $3$ and $6$, $4$ and $7$, $5'$ and $8$. By \cite{Vi11} there are also connections $0$ and $7$, $1$ and $8$, $2$ and $9$, $3$ and $10$, $4$ and $5$, $5'$ and $6$. So there are at most two indecomposable summands, as shown in the diagram:
$$
\xymatrix@R-1.5pc{
&&&&&{5}\ar@/_/@{-}[ld]\ar@/^/@{-}[rrd]\\
{0}\ar@/^3ex/@{-}[rr]&{1}\ar@/_3ex/@{-}[rr]&{2}\ar@/^3ex/@{-}[rr]&{3}\ar@/_/@{-}[rrd]&{4}&{}&{6}\ar@/^/@{-}[ld]\ar@/_3ex/@{-}[rr]&{7}\ar@/^3ex/@{-}[rr]&{8}\ar@/_3ex/@{-}[rr]&{9}&{10}\\
&&&&&{5'}
}
$$
Actually by \cite[Theorem~4.13]{Vi04} there are exactly two (isomorphic up to a Tate twist), for it is known that the first Witt index of $q$ is $2$.
\end{example}

\section{Applications to Chow motives: exceptional varieties}\label{sec-except}

For a split group $G$ we denote by $P_i$ the maximal parabolic subgroup of type $i$. The enumeration of simple roots follows Bourbaki. We denote by $\rho$ the coaction of $\Ch^*(G)$ on $\Ch^*(G/P_i)$.

\begin{ntt}[Variety $\E_7/P_7$, $p=2$]

We denote by $\E_7^{sc}$ the split simple simply connected group of type $\E_7$.

\begin{lem}
The Chow ring of $\E_7^{sc}$ together with the action of Steenrod operations\footnote{For Steenrod operations in arbitrary characteristic see \cite{Pr19}.} is as follows:
\begin{align*}
&\CH^*(\E_7^{sc})\otimes\ff_2\simeq\ff_2[e_3,e_5,e_9]/(e_3^2,e_5^2,e_9^2);\\
&S^2(e_3)=e_5;\  S^4(e_5)=e_9;\\
&\codim e_i=i.
\end{align*}
\end{lem}
\begin{proof}
The description of the Chow ring follows from \cite[Table~II]{Kac85}
and the Steenrod algebra action is described in \cite[Proposition~5.1]{IKT76}. Note that \cite{IKT76} deals with the singular cohomology instead of Chow rings. But the multiplicative structure and the action of the Steenrod operations on $\Ch^*(X)$ and on $\Ch^*(G)$, where $X$ denotes the variety of Borel subgroups of $\E_7^{sc}$,  can be
described in purely combinatorial terms (see \cite{DuZ07} and \cite[Section~5]{GPS16}) and, hence, does not depend on the choice of the base field $F$.

The Chow ring  $\CH^*(\E_7^{sc})\otimes\ff_2$ can be identified with the image of the natural
pullback homomorphism $H^{2*}_{\sing}(X,\ff_2) \to H^{2*}_{\sing}(\E_7^{sc}, \ff_2)$.
This image and formulae describing the action of Steenrod operations are described in \cite{IKT76} and \cite{MT78} and imply the formulae in the statement of the present lemma.
\end{proof}
 
\begin{lem}\label{E7P7mod2} $\CH^*(\E_7^{sc}/P_7)\otimes\ff_2\simeq\ff_2[h,x_5,x_9]/(h^{14},x_5^2,x_9^2)$ with $\codim h=1$ and $\codim x_i=i$,
and the coaction is given by
\begin{align}
&\rho(h)=1\otimes h;\label{ff1}\\
&\rho(x_5)=e_5\otimes 1+e_3\otimes h^2+1\otimes x_5;\label{ff2}\\
&\rho(x_9)=e_9\otimes 1+e_5\otimes h^4+1\otimes x_9.\label{ff3}
\end{align}
\end{lem}
\begin{proof}
The multiplicative structure is described in \cite[Theorem~5]{Du10}. (Alternatively one can compute the multiplicative structure using the algorithm described in \cite[Section~2]{PS10} or using the equivariant algorithm described in \cite[Section~5]{GPS16}).

By Lemma~\ref{chow-par} (applied to the case when $G=\E_7^{sc}$, $Q=P_7$ and $C=\E_6^{sc}$) the generator $x_5$ goes to $e_5$ and $x_9$ goes to $e_9$ under the natural map $\Ch^*(\E_7^{sc}/P_7)\to\Ch^*(\E_7^{sc})$.
Indeed, by \cite[Table~II]{Kac85} $\Ch^*(\E_6^{sc})=\ff_2[e_3]/(e_3^2)$ and, thus, the generators $e_5$ and $e_9$ in $\Ch^*(\E_7^{sc})$ map to $0$ under the natural homomorphism $\Ch^*(\E_7^{sc})\to \Ch^*(\E_6^{sc})$ and, therefore, come from $\Ch^*(\E_7^{sc}/P_7)$ by Lemma~\ref{chow-par}. But since $h\in\Ch^*(\E_7^{sc}/P_7)$ goes to $0$ in $\Ch^*(\E_7^{sc})$, by dimensional reasons $e_5$ and $e_9$ must come from $x_5$ and $x_9$ respectively.

Note that we can adjust $x_9$ by adding $x_5h^4$ if necessary so that $S^4(x_5)=x_9$.

Formula~\eqref{ff1} follows from the fact that $h$ is a rational cycle and from Lemma~\ref{comod-triv} applied to a generic $\E_7^{sc}$-torsor.

To prove formula~\eqref{ff2} we use the method described in Section~\ref{sec-coac}. Consider the parabolic subgroup $Q=P_1$ of our split group of type $\E_7$. Then the commutator subgroup $C$ of the Levi part of $Q$ is the group $\Spin_{12}$ of type $\D_6$, and by \cite[Table~II]{Kac85}
$$
\Ch^*(C)=\ff_2[e_3,e_5]/(e_3^2,e_5^2).
$$
Let $E$ be a generic $C$-torsor over $F$. The Chow motive of ${}_E(\E_7^{sc}/P_7)$ decomposes by \cite{Br05} as follows:
\begin{align}\label{dece7p7}
M({}_E(\E_7^{sc}/P_7))&=
M({}_E(\D_6/P'_1))\oplus M({}_E(\D_6/P'_1))\{17\}\oplus M({}_E(\D_6/P'_6))\{6\},
\end{align}
where $P'_i$'s on the right-hand side denote the respective maximal parabolic subgroups for the split group of type $\D_6$ (the enumeration of simple roots follows Bourbaki).

By Lemma~\ref{quad-eq} there is an element $y_5\in\Ch^5(\D_6/P'_1)$ such that
$$
\rho(y_5)=e_5\otimes 1+e_3\otimes h^2+1\otimes y_5,
$$
and under decomposition~\eqref{dece7p7} $y_5$ corresponds to either $x_5$ or $x_5+h^5$. Note that the homomorphism $\Ch^*(\E_7^{sc})\to\Ch^*(C)$ is an isomorphism in codimensions $\le 5$. Therefore, in both cases formula~\eqref{ff2} holds.

Finally, formula~\eqref{ff3} follows from formula~\eqref{ff2} by applying the Steenrod ope\-ra\-tion~$S^4$.
\end{proof}

The following proposition was previously obtained by Alexander Henke using classical methods.

\begin{prop}\label{dec-e7}
Let $E$ be an $\E_7^{sc}$-torsor over $\Spec F$ with $J_2(E)=(1,1,1)$. Then the Chow motive $M(E/P_7)$ with coefficients $\ff_2$ decomposes as $U(E/P_7)\oplus U(E/P_7)\{1\}$, where the motive $U(E/P_7)$ is indecomposable.
\end{prop}
\begin{proof}
By \cite[Theorem~3.5]{Ka13} the only possible summands of $M(E/P_7)$ up to Tate twists are upper motives of the form $U(E/P)$ for parabolic subgroups $P$ such that the Tits index of $(_E\E_7^{sc})_{F(E/P)}$ contains the Tits index of $({_E\E_7^{sc}})_{F(E/P_7)}$. Analyzing the Tits indices for groups of type $\E_7$ at the prime $2$ (see \cite[p.~59]{Ti66}) one can see that there are at most two candidates for possible summands of $M(E/P_7)$ up to Tate twist, namely $U(E/P_7)$ and the motive $R_2(E)$ (see Section~\ref{jinvold}).

However, using Theorem~\ref{RpE} and Lemma~\ref{E7P7mod2} we see that there are no summands of type $R_2(E)$.

We are going to apply the shell technique from \cite[Section~4]{GPS16}. To this end we need to compute the first shell. We have already established that $E/P_7$ is not generically split (see \cite[Theorem~5.7]{PS10}), so passing to its function field we get that the semisimple anisotropic kernel of $_E\E_7^{sc}$ over $F(E/P_7)$ is of type $\D_4$ or $\E_6$ and, therefore, since by \cite[Corollary~5.19]{PSZ08} the $J$-invariants of a group and of its semisimple anisotropic kernel are equivalent, $J_2(E_{F(E/P_7)})=(1,0,0)$.

Using Lemma~\ref{E7P7mod2} we see that the only cycles over $F(E/P_7)$ satisfying the relation $\rho(x)=1\otimes x$ are spanned by $h^i$, $x_9h^i$, $x_5h^{12}$, $x_5h^{13}$, $x_5x_9h^{12}$, $x_5x_9h^{13}$. On the other hand, it follows from the algorithm of \cite{CGM05} that the motive of $E/P_7$ contains over $F(E/P_7)$ exactly eight Tate motives, namely
$$
\ff_2,\,\ff_2\{1\},\,\ff_2\{9\},\,\ff_2\{10\},\,\ff_2\{17\}, \ff_2\{18\},\,\ff_2\{26\},\,\ff_2\{27\}.
$$
By the same arguments the only rational over $F$ cycles in $\Ch^9(\E_7^{sc}/P_7)$ are spanned by $h^9$, and this cycle is not Poincar\'e dual (we say that two cycles $\alpha$ and $\beta$ are Poincar\'e dual, if $\deg (\alpha\beta)=1$) to any of rational over $F(E/P_7)$ cycles in the dual codimension $18$, namely $0$, $x_9h^9$, $x_5h^{13}$ or $x_9h^9+x_5h^{13}$. Therefore, $h^9$ does not belong to the first shell, i.e., the Tate motive $\ff_2\{9\}$ from the list above is not a generic point of a direct summand of $M(E/P_7)$ over $F$.

Similarly, there are no rational over $F$ cycles in  $\Ch^{17}(\E_7^{sc}/P_7)$ and $\Ch^{26}(\E_7^{sc}/P_7)$ when $J_2(E)=(1,1,1)$. Therefore, the Tate motives $\ff_2\{17\}$ and $\ff_2\{26\}$ are also not generic points of a direct summand of $M(E/P_7)$ over $F$.

On the other hand, the cycle $h$, which is rational over $F$, belongs to the first shell.
It remains to apply \cite[Theorem~4.10]{GPS16} with $b=1$ and $\alpha=h$.
\end{proof}
\end{ntt}

\begin{ntt}[Variety $\E_8/P_8$, $p=3$]
We denote by $\E_8$ the split group of the respective type.

\begin{lem}
$\CH^*(\E_8)\otimes\ff_3\simeq\ff_3[e_4,e_{10}]/(e_4^3,e_{10}^3)$.
\end{lem}
\begin{proof}
Follows from \cite[Table~II]{Kac85}.
\end{proof}

\begin{lem}\label{E8mod3}
$$\CH^*(\E_8/P_8)\otimes\ff_3\simeq\ff_3[h,x_6,x_{10}]/(h^{20}-x_6^3h^2, x_6^4+h^{24}-x_6h^{18},x_{10}^3-x_6h^{24}-x_6^2h^{18}+x_6^5)$$ with $\codim h=1$ and $\codim x_i=i$,
and the coaction is given by
\begin{align}
&\rho(h)=1\otimes h;\label{g1}\\
&\rho(x_6)=e_4\otimes h^2+1\otimes x_6;\label{g2}\\
&\rho(x_{10})=e_{10}\otimes 1+e_4^2\otimes h^2-e_4\otimes x_6+1\otimes x_{10}.\label{g3}
\end{align}
\end{lem}
\begin{proof}
We proceed similar as in the proof of Lemma~\ref{E7P7mod2}.
The multiplicative structure is described in \cite[Theorem~7]{Du10}.
Formula~\eqref{g1} follows from Lemma~\ref{comod-triv}.

Consider now the parabolic subgroup $Q=P_{7,8}$ corresponding to the subset $\{1,2,3,4,5,6\}$ of the Dynkin diagram of type $\E_8$. Then the commutator subgroup $C$ of the Levi part of $Q$ is of type $\E_6^{sc}$, and by \cite[Table~II]{Kac85}
$$
\Ch^*(C)=\ff_3[e_4]/(e_4^3).
$$
Let $E$ be a generic $C$-torsor over $F$. The Chow motive of ${}_E(\E_8^{sc}/P_8)$ decomposes by \cite{Br05} and \cite{CGM05} as follows:
\begin{equation}\label{decE8}
\begin{aligned}
M({}_E(\E_8/P_8))&=\ff_2\oplus \ff_2\{1\}\oplus \ff_2\{28\}\oplus \ff_2\{29\}\\
&\oplus \ff_2\{56\}\oplus \ff_2\{57\}\oplus M({}_E(\E_6^{sc}/P'_6))\{2\}\\
&\oplus M({}_E(\E_6^{sc}/P'_6))\{29\}\oplus M({}_E(\E_6^{sc}/P'_6))\{30\}\oplus M({}_E(\E_6^{sc}/P'_1))\{11\}\\
&\oplus M({}_E(\E_6^{sc}/P'_1))\{12\}\oplus M({}_E(\E_6^{sc}/P'_1))\{39\}\oplus M({}_E(\E_6^{sc}/P'_2))\{18\},
\end{aligned}
\end{equation}
where $P'_i$'s on the right-hand side denote the respective maximal parabolic subgroups for the split simply connected group of type $\E_6$ (the enumeration of simple roots follows Bourbaki).

Similar to the proof of Lemma~\ref{E7P7mod2} we use Lemma~\ref{chow-par} (applied to $G=\E_6^{sc}$, $Q=P'_6$ and $C=\D_5^{sc}$) to obtain an element $y_4\in\Ch^4(\E_6^{sc}/P'_6)$ mapping to $e_4$, so
$$
\rho(y_4)=e_4\otimes 1+1\otimes y_4.
$$
Under the decomposition above $1$ corresponds to $h^2$ and $y_4$ corresponds to either $x_6$ or to $x_6\pm h^6$. In both cases formula~\eqref{g2} holds.

Furthermore, we have
$$
\rho(y_4^2)=e_4^2\otimes 1-e_4\otimes y_4+1\otimes y_4^2.
$$
Under the decomposition above $y_4^2$ corresponds to $z=x_{10}+\alpha x_6h^4+\beta h^{10}$ for some $\alpha,\beta\in\ff_3$. Note that the coefficient of $x_{10}$ in $z$ is non-zero, since otherwise we would come to a contradiction with formula~\eqref{g2}. Namely, we would get $$\alpha(e_4\otimes h^6+1\otimes x_6h^4)+\beta\otimes h^{10}=\rho(\alpha x_6h^4+\beta h^{10})=e_4^2\otimes h^2+\ldots\mod (e_{10})\otimes\Ch^*(\E_8/P_8)$$ which is impossible.

It follows that
\begin{align}\label{qewe8}
\rho(x_{10})=e_4^2\otimes h^2-e_4\otimes x_6+1\otimes x_{10}+\gamma e_4\otimes h^6\mod (e_{10})\otimes\Ch^*(\E_8/P_8)
\end{align}
for some $\gamma\in\ff_3$, where $(e_{10})$ denotes the ideal of $\Ch^*(\E_8)$ generated by $e_{10}$.

We have a commutative diagram
\begin{equation}
\xymatrix{
\Ch^*(\E_8/P_8)\ar[d]^-{\pi}\ar[r]^-\rho&\Ch^*(\E_8)\otimes\Ch^*(\E_8/P_8)\ar[d]^-{\id\otimes\pi}\\
\Ch^*(\E_8)\ar^-{\Delta}[r]&\Ch^*(\E_8)\otimes\Ch^*(\E_8)
}
\end{equation}

The image of $x_{10}$ in $\Ch^*(\E_8)$ is $e_{10}$, so
\begin{equation}\label{qewe2}
\rho(x_{10})=e_{10}\otimes 1+1\otimes x_{10}\mod \Ch^*(\E_8)\otimes \Ker\pi.
\end{equation}

Moreover, $\Ker\pi=(h,x_6)$.
Changing $x_{10}$ to $x_{10}-\gamma x_6h^4$ (this does not affect the relations) and combining formulae~\eqref{qewe8} and \eqref{qewe2} we get formula~\eqref{g3}.

\end{proof}

\begin{prop}
Let $E$ be an $\E_8$-torsor over $\Spec F$ with $J_3(E)=(1,1)$. Then the Chow motive $M(E/P_8)$ with coefficients $\ff_3$ decomposes as $$
U(E/P_8)\oplus U(E/P_8)\{1\}\oplus\bigoplus_{i=4}^{25}R_3(E)\{i\},
$$
where the motive $U(E/P_8)$ is indecomposable and $R_3(E)$ is the upper motive of $E/B$.
\end{prop}
\begin{proof}
We proceed similar as in the proof of Proposition~\ref{dec-e7}.

By \cite[Theorem~3.5]{Ka13} the only possible summands of $M(E/P_8)$ up to Tate twists are upper motives of the form $U(E/P)$ for parabolic subgroups $P$ such that the Tits index of $({_E\E_8})_{F(E/P)}$ contains the Tits index of $({_E\E_8})_{F(E/P_8)}$. Analyzing the Tits indices for groups of type $\E_8$ at the prime $3$ (see \cite[p.~60]{Ti66}) one can see that there are at most two candidates for possible summands of $M(E/P_8)$ up to Tate twist, namely $U(E/P_8)$ and the motive $R_3(E)$ (see Section~\ref{jinvold}).

Using Lemma~\ref{E8mod3} we see that in Theorem~\ref{RpE} all possible $\beta$'s are of the form $x_6^2x_{10}^2h^j$ for $j=0,\ldots,21$, while the corresponding $\alpha$'s are $h^{j+4}$. Note that $e_J=e_4^2e_{10}^2$ and $j$ runs till $21$, since $h^{4+21}$ is not zero, but the next power $h^{4+22}=0$ (this immediately follows from the relations of Lemma~\ref{E8mod3}). This allows to split off the summands $R_3(E)$ as in the statements of the proposition.

We are going to apply the shell technique from \cite[Section~4]{GPS16}.
Note that the motivic decomposition of $E/P_8$ over its function field is exactly the same as decomposition~\eqref{decE8}, since by \cite[Theorem~5.7]{PS10} the variety $E/P_8$ is not generically split and, therefore, by the classification of Tits indices at the prime $3$ (see \cite{Ti66}) corresponds a semisimple anisotropic kernel of type $\E^{sc}_6$.
In particular, there are six Tate motives in this decomposition exactly as in formula~\eqref{decE8}.

Moreover, the rank of every motivic direct summand of $E/P_8$ is divisible by $p=3$, since otherwise $E/P_8$ would contain a zero-cycle of degree $1$ mod $3$ over $F$, and this is impossible, since $J_3(E)=(1,1)$ by our assumption.

On the other hand, the cycle $h$, which is rational over $F$, belongs to the first shell.
It remains to apply \cite[Theorem~4.10]{GPS16} with $b=1$ and $\alpha=h$.
\end{proof}
\end{ntt}

\begin{rem}
Let $P$ be a parabolic subgroup of a split group of type $\E_8$.
It follows immediately from the classification of the Tits indices for groups of type $\E_8$
that the motives $R_3(E)$ and $U(E/P_8)$ from the proposition above are (up to Tate twists) the only possible indecomposable motivic summands of every twisted flag variety of the form $E/P$ at the prime $3$, when $J_3(E)=(1,1)$.

For the case $J_3(E)=(1,0)$ see \cite[Section~10c]{GPS16}.
\end{rem}

\begin{ntt}[Case $\E_8$, $p=2$]

\begin{lem}
We have $\CH^*(\E_8)\otimes\ff_2\simeq\ff_2[e_3,e_5,e_9,e_{15}]/(e_3^8,e_5^4,e_9^2,e_{15}^2),$
where ${\codim e_i=i}$, $e_3$, $e_5$ and $e_9$ are primitive while
\begin{equation}\label{e15}
\Delta(e_{15})=e_{15}\otimes 1+e_9\otimes e_3^2+e_5\otimes e_5^2+e_3\otimes e_3^4+1\otimes e_{15}.
\end{equation}
\end{lem}
\begin{proof}
See \cite[Theorem~6.32]{MT78}. 
Alternatively one can consider a commutative diagram analogous to diagram~\eqref{diagso} with the group $\E_8$ instead of $\SO_n$ and with $\E_8/P_8$ instead of $\SO_n/P_1$. Then the generator $e_{15}$ lies in the image of the respective homomorphism $\pi$, and one can recover the formulae for the coproduct $\Delta(e_{15})$ using this diagram and formulae for the coaction $\rho$, which can be obtained similar as in Lemma~\ref{E8mod3}.
\end{proof}

\begin{prop}
If $J_2(E)=(*,*,*,0)$, then $J_2(E)\le (1,1,1,0)$ or $J_2(E)\le (2,1,0,0)$ component-wisely.
\end{prop}
\begin{proof}
The condition on $J_2(E)$ means that there is a rational element in $\Ch^*(\E_8/B)$ mapping to $x=e_{15}+\alpha e_5^3+\beta e_3^5+\gamma e_3^2 e_9\in\Ch^*(\E_8)$ for some $\alpha,\,\beta,\,\gamma\in\ff_2$. We have
\begin{align*}
\Delta(x)&=x\otimes 1+1\otimes e_{15}+\alpha e_5^2\otimes e_5+(1+\alpha)e_5\otimes e_5^2+\alpha \otimes e_5^3\\
&+\beta e_3^4\otimes e_3+(1+\beta)e_3\otimes e_3^4+\beta \otimes e_3^5\\
&+\gamma e_3^2\otimes e_9+(1+\gamma)e_9\otimes e_3^2+\gamma \otimes e_3^2e_9.
\end{align*}
Since $J$ is a bi-ideal by Lemma~\ref{bi-ideal}, it follows that $e_5^2$, $e_3^4$ and either $e_3^2$ or $e_9$ belong to $J$, as claimed.
\end{proof}
\end{ntt}

\section{Applications to other cohomology theories}\label{secother}

Let $M$ be a Chow motive and let $A^*$ be an oriented cohomology theory. By \cite[Section~2]{VY07} there is a unique lift of the motive $M$ to the category of $\Omega^*$-motives and,
since $\Omega^*$ is the universal oriented cohomology theory, there is a respective motive in the category of $A^*$-motives. This allows to consider every Chow motive $M$ also as an $A^*$-motive for an arbitrary oriented cohomology theory $A^*$.

In the following examples we will, in particular, compare the behaviour of $A^*$-motives with Chow motives for different oriented cohomology theories $A^*$.

In all our examples all graded modules of constant rank over $A^*(\pt)$ are free. Moreover, $H^*$ are also free over $A^*(\pt)$. Thus, by Lemma~\ref{cofree} the second assumption of Theorem~\ref{mainthm} is satisfied.

\begin{example}
Let $p$ be a prime number. By \cite[Corollary~5.11]{Me97} one has the following isomorphism of rings
$$K^0[\beta,\beta^{-1}](\PGL_p)\otimes\ff_p\simeq\ff_p[\beta,\beta^{-1}][x]/(x^p)=:H^*$$
with $\deg x=1$. The coproduct structure is given by $\Delta(x)=x\otimes 1+1\otimes x-\beta x\otimes x$.

In particular, the dual algebra $H^\vee$
is isomorphic to $\ff_p[\beta,\beta^{-1}][y]/(y^p-\beta^{p-1}y)$.
Since ${y^p-\beta^{p-1}y=\prod_{j=0}^{p-1}(y-j\beta)}$ and the polynomials $y-j\beta$ are coprime, we get that there are exactly $p$  non-isomorphic types of indecomposable direct summands of the $K^0$-motive of the respective generically split varieties.

Besides, if we denote by $M_j$ the $H^\vee$-module corresponding to the $j$-th polynomial $y-j\beta$, then we have an isomorphism of $H^\vee$-modules $M_i\otimes_{\ff_p[\beta,\beta^{-1}]}M_j\simeq M_{i+j \bmod p}$. Moreover, the module $M_0$ corresponds to the Tate motive.

This agrees with the result of Quillen with a computation of the $K$-theory of Severi--Brauer varieties (see \cite[Section~8]{Qui73}, see also \cite{Pa94}). We remark also that contrary to the $K^0$-case the Chow motive modulo $p$ of the Severi--Brauer variety $\SB(A)$, where $A$ is a central simple division algebra of degree $p$, is indecomposable.
\end{example}

\begin{example}
Let $G$ be a split semisimple simply-connected algebraic group and $p$ a prime number. Then by \cite{Me97} we have $K^0[\beta,\beta^{-1}](G)\otimes\ff_p=\ff_p[\beta,\beta^{-1}]$. In particular, for every $G$-torsor $E$ the respective bialgebra $H^*=\ff_p[\beta,\beta^{-1}]$. Therefore, the $K^0$-motive of $E/B$ is a direct sum of Tate motives. This agrees with \cite{Pa94}. 
\end{example}

\begin{example}\label{typeone}
Let $(G,p)$ be from the following list: $(\G_2,2)$, $(\F_4,2)$,
$(\E_6,2)$, $(\F_4,3)$, $(\E_6^{sc},3)$, $(\E_7,3)$, or $(\E_8,5)$,
where $sc$ stands for the split simply connected group (in the cases $(\E_6,2)$ and $(\E_7,3)$ one can choose an arbitrary split group of the respective type).
For the localized at $p$ algebraic cobordism one has by \cite[Theorem~5.1]{Ya05} the following isomorphism of rings
$$\Omega^*_{(p)}(G)\simeq \laz_{(p)}[x_{p+1}]/(px_{p+1}, v_1x_{p+1},x_{p+1}^p)$$
with $\deg x_{p+1}=p+1$.

In particular, the second Morava $K$-theory modulo $p$ equals $$K(2)^*(G)\simeq\ff_p[v_2,v_2^{-1}][x_{p+1}]/(x^p_{p+1})=:H^*.$$

In fact, by dimensional reasons (since $\deg v_2=-(p^2-1)$ and $\deg x_{p+1}=\deg \Delta(x_{p+1}) =p+1$)
and by the coassociativity of the coproduct, the coproduct structure on $H^*$ must be given by
\begin{align}\label{f1}
\Delta(x_{p+1})=x_{p+1}\otimes 1+1\otimes x_{p+1}+\alpha v_2\sum_{i=1}^{p-1}\frac 1p\binom{p}{i}x_{p+1}^i\otimes x_{p+1}^{p-i}
\end{align}
for some scalar $\alpha\in\ff_p$ (note that the coefficients $\tfrac 1p\binom{p}{i}$ are integers). Indeed, there are no higher degree terms, since $x_{p+1}^p=0$, and the coefficients of $x_{p+1}^i\otimes x_{p+1}^{p-i}$ in the sum are determined by the equation $(\id\otimes\Delta)\circ\Delta=(\Delta\otimes\id)\circ\Delta$ on the coproduct. Note that to solve this equation is a finite task, since $p=2$, $3$ or $5$. The obtained system of equations on the coefficients has a unique solution up to a scalar $\alpha$.

Note also that the right-hand side of formula~\eqref{f1} gives the first $p+1$ terms of the formal group law for the Morava $K$-theory (this is not surprising, since the coassociativity of the coproduct translates into the associativity of the formal group law). 

Note that for Chow rings modulo $p$ one has $\Ch^*(G)\simeq \ff_p[x_{p+1}]/(x_{p+1}^p)$, where the generator $x_{p+1}$ is a primitive element, and the respective indecomposable Chow motives are the generalized Rost motives (see \cite[Section~7]{PSZ08}).
Note also that by \cite[Proposition~6.2]{SeS20} these Rost motives are decomposable with respect to $K(2)^*$.

The dual algebra $H^\vee$ is isomorphic to $\ff_p[v_2,v_2^{-1}][y]/(y^p-\alpha v_2y)$.
Since it is indecomposable as an $H^\vee$-module for $\alpha=0$ (this can be seen directly or using Theorem~\ref{mainthm} and the fact that the respective Chow motive is indecomposable) and since the respective motive for $K(2)^*$ is decomposable, the scalar $\alpha$ must be non-zero.

In this case $y^p-\alpha v_2y=y(y^{p-1}-\alpha v_2)$ and the respective Rost motive for $K(2)^*$ modulo $p$ decomposes into a direct sum of two non-isomorphic indecomposable motives, one of which is the Tate motive (this motive corresponds to the $H^\vee$-module $\ff_p[v_2,v_2^{-1}][y]/(y)$). This agrees with \cite[Proposition~6.2]{SeS20}.
\end{example}

\begin{rem}
In Example~\ref{typeone} we do not use the full generality of \cite[Proposition~6.2]{SeS20}, but just the fact that the respective generalized Rost motives are decomposable with respect to the second Morava $K$-theory. But this can be seen directly. Indeed, if $X$ is a norm variety (see \cite[Definition~4.1]{S16}) of dimension $p^2-1$, then the projector $v_2^{-1}\cdot (1\times 1)\in K(2)^*(X\times X)$ defines a direct summand of the respective generalized Rost motive. This direct summand is isomorphic to the Tate motive.
\end{rem}

\begin{example}[Rost motives]
Let $p=2$ and consider the $n$-th Morava $K$-theory $K(n)$.
We consider the group $G=\SO_{2^{n+1}}$ and a $G$-torsor $E$ corresponding to an anisotropic $(n+1)$-fold Pfister form $q$.

Then the respective indecomposable Chow motive is the Rost motive associated with $q$, and as in Example~\ref{typeone} we have
$H^*\simeq\ff_2[v_n,v_n^{-1}][x]/(x^2)$ with $\deg x=2^{n}-1$. The coproduct is given by $\Delta(x)=x\otimes 1+1\otimes x+v_n x\otimes x$ and the dual algebra $$H^\vee\simeq \ff_2[v_n,v_n^{-1}][y]/(y^2+v_ny).$$ In particular, since $y^2+v_ny=y(y+v_n)$ the respective Rost motive for the $n$-th Morava $K$-theory is a direct sum of two non-isomorphic motives, one of which is the Tate motive. This agrees with \cite[Proposition~6.2]{SeS20}.
\end{example}

\begin{example}
For the localized at $p=3$ algebraic cobordism one has by \cite[Theorem~5.2]{Ya05} the following isomorphism of rings
$$\Omega^*(\E_8)_{(3)}\simeq \laz_{(3)}[x_4,x_{10}]/(3x_4, 3x_{10}, x_4^3, x_{10}^3, v_1x_4+v_2x_{10},v_1x_{10})$$
with $\deg x_i=i$. In particular, the second Morava $K$-theory of $\E_8$ modulo $p=3$ equals
$$K(2)^*(\E_8)\simeq\ff_3[v_2,v_2^{-1}][x_4]/(x_4^3)=:H^*.$$

The coproduct structure is given again by formula~\eqref{f1},
and the scalar $\alpha$ is non-zero, since it follows from \cite[Theorem~5.7]{PS10} that there exists a field extension of the base field over which our variety of Borel subgroups of type $\E_8$ decomposes into a direct sum of generalized Rost motives modulo $3$.

Note that the respective indecomposable Chow motive modulo $3$ has Poincar\'e polynomial $\dfrac{t^{12}-1}{t^4-1}\cdot \dfrac{t^{30}-1}{t^{10}-1}$ and, in particular, rank $9$. Contrary to this, the respective $K(2)^*$-motive has rank $3$.

As in Example~\ref{typeone} the dual algebra is isomorphic to $\ff_p[v_2,v_2^{-1}][y]/(y^3-\alpha v_2y)$, and the respective $K(2)^*$-motive decomposes further as a direct sum of two non-isomorphic indecomposable motives, one of which is the Tate motive.
\end{example}

\begin{rem}
Motivic decompositions which we considered in this article were usually with modulo $p$ coefficients. Nevertheless, there is a standard technique to lift motivic isomorphisms and motivic decompositions from $\ff_p$- to $\zz_{(p)}$- or $\zz_p$-coefficients (see e.g. \cite{SZ15}).
\end{rem}

\begin{rem}
Let $A^*\to B^*$ be a morphism between two oriented cohomology theories.
Vishik and Yagita provide in \cite[Section~2]{VY07}
a criterion under what conditions
there is a one-to-one correspondence between the isomorphism classes of $A^*$-motives
and the isomorphism classes of $B^*$-motives. This allows to extend our results to a bigger class of oriented cohomology theories.
\end{rem}

\begin{rem}
Let $G$ be a split semisimple algebraic group, $E$ a generic $G$-torsor and $P$ a special parabolic subgroup of $G$.
Let $A^*$ be a free oriented cohomology theory.
Due to nilpotency results \cite[Section~5]{CNZ19} (cf. \cite[Theorem~5.5]{PS17}, \cite{NPSZ18}) one can lift motivic decompositions of the $A^*$-motives of twisted flag varieties $E/P$ to a motivic decomposition of the $G$-equivariant $A^*$-motive of $G/P$.

In particular, the results of the present article provide new motivic decompositions for equivariant motives.
\end{rem}

\end{document}